\documentclass[twoside,11pt]{article}

\usepackage{algorithm2e}
\usepackage{HL_wp_article}
\usepackage{hyperref}
\RequirePackage{amsmath}
\DeclareMathAlphabet{\mathpzc}{OT1}{pzc}{m}{it}
\usepackage{enumitem}
\newtheorem{assumption}{Assumption}
\usepackage{setspace}
\newenvironment{itemize*}%
  {\begin{itemize}%
    \setlength{\itemsep}{0pt}%
    \setlength{\parskip}{0pt}}%
  {\end{itemize}}

\jmlrheading{}{}{}{}{}{}{}

\ShortHeadings{High-Dimensional Stochastic Optimization Under Low-Rankness}{Liu, Hernandez, and Lee}
\firstpageno{1}

\begin{document}
\sloppy
\title{Regularized Sample Average Approximation for High-Dimensional Stochastic Optimization Under Low-Rankness}

\author{\name Hongcheng Liu \email liu.h@ufl.edu \\
\name Charles Hernandez \email cdhernandez@ufl.edu \\
\name Hung Yi Lee \email hungyilee@ufl.edu \\
       \addr Department of Industrial and Systems Engineering\\
       University of Florida\\
       Gainesville, FL 32611, USA
      }

\editor{ }
\singlespacing
\maketitle


\begin{abstract} 
This paper concerns a high-dimensional stochastic programming problem of minimizing a function of expected cost with a matrix argument. To this problem, one of the most widely applied solution paradigms is the sample average approximation (SAA), which uses  the average cost over  sampled scenarios as a surrogate to approximate the expected cost. 
Traditional   SAA theories require the sample size to grow rapidly when the problem dimensionality increases. Indeed, for a problem of optimizing over a $p$-by-$p$ matrix, the sample complexity of the SAA is given by $\tilde O(1)\cdot \frac{p^2}{\epsilon^2}\cdot{polylog}(\frac{1}{\epsilon})$ to achieve an $\epsilon$-suboptimality gap, for some poly-logarithmic function ${polylog}(\,\cdot\,)$ and some quantity $\tilde O(1)$ independent of dimensionality $p$ and sample size $n$.  In contrast, this paper considers a regularized SAA (RSAA) with a low-rankness-inducing penalty.  We demonstrate that the sample complexity of RSAA is  $\tilde O(1)\cdot \frac{p}{\epsilon^3}\cdot {polylog}(p,\,\frac{1}{\epsilon})$, which is almost linear in $p$ and thus indicates a substantially lower dependence on dimensionality. Therefore,   RSAA can be more advantageous than SAA especially for larger scale and higher dimensional problems.  Due to the close correspondence between stochastic programming and statistical learning, our results also indicate that high-dimensional low-rank matrix recovery is possible generally beyond a linear model, even if the common assumption of restricted strong convexity is completely absent.
\end{abstract}

\bigskip
\clearpage

\section{Introduction}

As dimensionality inflates in modern applications of stochastic programming (SP) in order to generate more comprehensive and higher-granular  decisions,  the sample average approximation (SAA), which is traditionally a  common solution paradigm for  SP, sometimes tends to be  demanding for sample availability. The current SAA theories as per \cite{Shapiro:book}, \cite{Shapiro2003a}, \cite{Shapiro2003b} and \cite{Shapiro2007} require that the number of samples should always be greater than  the number of decision variables; for optimizing over a $p$-by-$p$ matrix, the sample size $n$ should grow at least quadratically in $p$. 
Such  sample size requirement may be undesirably costly in certain high-dimensional applications. Recently,   a regularized SAA with sparsity-inducing penalty has been studied by  \cite{Liu2018}, which shows that  significant reduction of sample size requirement may be achieved by exploiting sparse structures in the problem. This current paper then seeks to generalize the result therein to the settings where sparsity is replaced by a low-rankness assumption. We will show that a similar level of success can be achieved.

The particular problem of focus is stated as follows: Let   $Z\in \mathcal W$, for some $\mathcal W\subseteq\Re^q$ and  $q>0$, be a random vector. 
 Consider a measurable, deterministic function $ f:\,\mathcal S_{p}\times \mathcal W\rightarrow \Re$ where $\mathcal S_p$ is the cone of $p$-by-$p$ ($p\geq 1$) symmetric and positive semidefinite matrices and  $f(\mathbf X, Z)$ is a cost function with respect to parameter $Z$ and a fixed matrix of decision variables ${\mathbf X}$.  Then the problem of consideration is an SP problem  given as
 \begin{align}
{\mathbf X}^*\in\,{\arg\min}\,\left\{\mathbb F({\mathbf X}):\,\mathbf X\in \mathcal S_p \right\}.\label{original problem}
\end{align}
where $\mathbb F({\mathbf X})=\mathbb E[f(\mathbf X,Z)]$ is well-defined and finite-valued for any given $\mathbf X\in\mathcal S_p$. Assume, hereafter, that $\sigma_{\max}({\mathbf X}^*)\leq R$ for some constant $R\geq 1$, where $\sigma_{\max}(\cdot)$ denotes the spectral radius. With some abuse of terminology, we say that the dimensionality of this problem is $p$, since the unknown is a $p$-by-$p$ matrix.
 We refer to this optimization problem as the ``true problem'' and $\mathbf X^*$ as the ``true solution'', as they assume the exact knowledge of the underlying distribution and the admissibility of calculating the multi-dimensional integration involved in evaluating the expected cost.   We would like to remark that \eqref{original problem} subsumes the unconstrained problems since any symmetric matrix can be represented by the difference between two symmetric and positive semidefinite matrices. Furthermore, also subsumed by  \eqref{original problem} are problems with  non-symmetric and non-square matrices $\mathbf X$,   since  they can be transformed into symmetric matrices by the self-adjoint dilation with $\bar{\mathbf X}=\left[\begin{matrix}\mathbf 0&\mathbf X\\ \mathbf X^\top&\mathbf 0\end{matrix}\right]$ for some all-zero matrices $\mathbf 0$'s with proper dimensions. 

Hereafter, let  $\mathbf Z_1^n=(Z_1,...,Z_n)$ be a sequence of $n$-many i.i.d.\,random samples of $Z$. To solve Problem \eqref{original problem}, one of the most popular solution schemes, as mentioned above, is to invoke the following SAA formulation as a surrogate:
 \begin{align}
{\mathbf X}^{SAA}\in\,{\arg\min}\,\left\{\mathcal F_n({\mathbf X},\,\mathbf Z_1^n):=\frac{1}{n}\sum_{i=1}^nf(\mathbf X, Z_i):\,\mathbf X\in \mathcal S_p \right\}.\label{original problem 2}
\end{align}
According to the seminal results by \cite{Shapiro:book}, ${\mathbf X}^{SAA}$ well approximates $\mathbf X^*$ in the sense that 
\begin{align}
\mathbb F({\mathbf X}^{SAA})-\mathbb F(\mathbf X^*)\leq \tilde O(1)\cdot\sqrt{\frac{p^2\cdot\ln n}{n}}\label{traditional bound}
\end{align}
with high probability, where $\tilde O(\cdot)$ is some quantity  that is independent of  $p$ and $n$. Thus, to ensure the same suboptimality gap, it stipulates that the sample size, $n$, must grow qradratically if  $p$ increases. For an SP problem where ${\mathbf X}^*$ is sparse and $f$ is twice-differentiable almost surely, we have shown in \cite{Liu2018} that \eqref{traditional bound} can be  sharpened, in terms of its dependence on $p$, into:
\begin{align}
\mathbb F(\mathbf X^{RSAA})-\mathbb F(\mathbf X^*)\leq \tilde O(1)\cdot {\frac{\sqrt{\ln (np)}}{n^{1/4}}}\label{sparsity bound},
\end{align}
with high probability, where $\mathbf X^{RSAA}$ is an SAA scheme with sparsity-inducing regularization. Similar (and potentially stronger) results than the above  have been reported by  \cite{Liu2019b} and \cite{Liu2019} in the context of high-dimensional statistical and machine learning under a sparsity assumption and/or its limited variations.

In contrast, this paper  provides a substantial generalization to  \cite{Liu2018,Liu2019b,Liu2019} by weakening the sparsity and twice-differentiability assumptions simultaneously to low-rankness and continuous differentiability. Particularly, our low-rankness assumption is as below:
\begin{assumption}\label{approximate low-rankness}
The rank ${\bf rk}(\,\cdot\,)$ of ${\mathbf X}^*$  in the problem \eqref{original problem} satisfies $s:=\mathbf{rk}({\mathbf X}^*)\ll p$ for some $s\geq 1$. 
 \end{assumption}
The above low-rankness assumption is more general than the sparsity assumption of a vector, since any vector $\boldsymbol x$ can be represented by a diagonal matrix, $diag(\boldsymbol x)$, whose diagonal entries equal to $\boldsymbol x$. Then, sparsity of $\boldsymbol x$ implies that $diag(\boldsymbol x)$ is of low rank.  
Furthermore, we generalize the assumption twice-differentiability to   Lipschitz continuity of the partial derivatives of $f$ w.r.t.\ the eigenvalues of the input matrix, as we will discuss in more detail subsequently.
 
  For this more general problem, our solution paradigm modifies the SAA into the following regularized SAA (RSAA):
 \begin{align}
{\mathbf X}^{RSAA}\in\,\underset{\mathbf X\in\mathcal S_p}{\arg\min}\,\left\{\mathcal F_{n,\lambda}({\mathbf X},\,\mathbf Z_1^n)\vphantom{\frac{1}{n}}:=\mathcal F_{n}({\mathbf X},\,\mathbf Z_1^n)+\sum_{j=1}^{p}P_\lambda( \sigma_j(\mathbf X))\right\},\label{original problem 2}
\end{align}
where $ \sigma_j(\mathbf X)$ stands for the $j$th eigenvalue of $\mathbf X$ and $P_\lambda$ is a penalty function in the form of the minimax concave penalty (MCP) \citep{Zhang2010} given as $P_\lambda(x)=\int_{0}^x\frac{[a\lambda-t]_+}{a}dt$, for some user-specific tuning parameters $a,\,\lambda>0$. Here $[\,\cdot\,]_+=\max\{\,\cdot\,, 0\}$. The MCP is a mainstream special form of the folded concave penalty (FCP) first proposed by \cite{Fan2001}.


Under the above settings, the RSAA formulation is nonconvex and its global solutions are elusive. To ensure computability, this paper considers stationary points that satisfy a set of significant subspace second-order necessary conditions (S$^3$ONC), given as in Definition \ref{SONC definition} in the subsequent. The S$^3$ONC herein is an extension to a similar notion presented by \cite{Liu2017, Liu2018} and  is a special case than the canonical second-order KKT conditions. Hence, any second-order (local optimization) algorithm that computes a second-order KKT solution ensures the S$^3$ONC. The resulting computational effort of an S$^3$ONC solution (a solution that satisfies the S$^3$ONC) is likely tractable.

Let $ {\mathbf X}^{\ell_1}_{\lambda}$ be defined as 
\begin{align}
    {\mathbf X}^{\ell_1}_{\lambda}\in {\arg\min}_{{\mathbf X}\in\mathcal S_p}\, \mathcal F_n({\mathbf X},\, \mathbf Z_1^n)+\lambda\Vert {\mathbf X} \Vert_*,\label{nuclear problem}
 \end{align}
  with $\Vert \cdot \Vert_*$ denoting the nuclear norm.
 We show that, under a few standard assumptions in addition to Assumptions \ref{approximate low-rankness}, for any S$^3$ONC solution  to the RSAA, denoted  $\mathbf X^{RSAA}$, which satisfies $\mathcal F_{n,\lambda}(\mathbf X^{RSAA})\leq \mathcal F_{n,\lambda}( {\mathbf X}^{\ell_1}_{\lambda})$ a.s., it holds that 
\begin{align}
\mathbb F({\mathbf X}^{RSAA})-\mathbb F(\mathbf X^*)\leq \tilde O(1)  \cdot\left(\frac{s\cdot p^{2/3}}{n^{2/3}}+\frac{s\cdot p^{1/3}}{n^{1/3}}\right)\cdot \ln(np)
\label{new log bound},
\end{align}
with overwhelming probability, when our knowledge on the rank of $\mathbf X^*$ is completely absent.  Furthermore, if we allow the penalty parameter to incorporate knowledge on the rank $s$ of $\mathbf X^*$, as  in Assumption \ref{approximate low-rankness}, then a better choice of $\lambda$ allows that
\begin{align}
\mathbb F({\mathbf X}^{RSAA})-\mathbb F(\mathbf X^*)\leq \tilde O(1)  \cdot\left(\frac{s^{1/3}\cdot p^{2/3}}{n^{2/3}}+\frac{s^{2/3}\cdot p^{1/3}}{n^{1/3}}\right)\cdot \ln(np),
\label{new log bound 2}
\end{align}
with overwhelming probability, where  the sample size requirement has a lower dependence on $s$ of $\mathbf X^*$ compared to \eqref{new log bound}. 
  The above results are then the promised, almost linear, sample complexity;  from both \eqref{new log bound} and \eqref{new log bound 2}, $n$ should only increase almost linearly in $p$ to compensate the growth in dimensionality.  This indicates that the RSAA   would be much more advantageous than the SAA especially for problems with higher dimensions. To compute the desired solution ${\mathbf X}^{RSAA}$, one may invoke an S$^3$ONC-guaranteeing algorithm initialized at $ {\mathbf X}^{\ell_1}_{\lambda}$. Meanwhile, the initial solution, $ {\mathbf X}^{\ell_1}_{\lambda}$, is often polynomial-time computable when $f(\,\cdot,\,w)$ is convex for almost every $w\in\mathcal W$ (although the convexity of $f(\,\cdot,\,w)$ is not necessary to prove the almost linear sample complexity).

To our knowledge, our paper presents the first SAA variant that ensures a sample complexity that is  almost linear in dimensionality under low-rankness.  Even though similar   results have been achieved previously, e.g., by  \cite{Negahban2012}, \cite{Rohde and Tsybakov 2011} and \cite{Elsener and van de geer 2018} in the context of  high-dimensional low-rank matrix estimation, most of the existing  results assume the presence of restricted strong convexity (RSC) or its variations. While the RSC is deemed generally plausible for statistical and/or machine learning, such type of assumptions are often  not satisfied by stochastic programming. Furthermore, due to the correspondence between the SAA and matrix estimation problems,  our results may also imply that high-dimensional matrix estimation is generally possible under the low-rankness assumption; even if the conditions such as the RSC or alike are completely absent, the MCP-based regularization may still ensure a sound generalization error as measured by the excess risk, which coincides in  formulation with the suboptimality gap in minimizing the SP. In addition, our results do not assume a linear or generalized linear model in data generation. 
 Even though a few other likely more important error bounds  are unavailable herein but are presented by \cite{Negahban2012}, \cite{Rohde and Tsybakov 2011} and \cite{Elsener and van de geer 2018} (most of whom  focus more on linear or generalized linear models under RSC or alike), we believe that the  excess risk is still an important out-of-sample performance measure commonly employed by, e.g., \cite{excess risk}, \cite{Koltchinskii 2010}, and \cite{Clemencon}. 

The rest of the paper is organized as follows: Section \ref{main results} presents our assumptions and main results. Section \ref{General proof} presents the general road map for our proof and major schemes employed. Section \ref{conclusions} then concludes our paper. All technical proofs are presented in the appendix.
\subsection{Notations}
Throughout this paper, we denote by $\Vert\cdot \Vert$ the 2-norm of a vector, by $\sigma_{\max}(\cdot)$ the spectral norm, by $\Vert\cdot\Vert_*$ the nuclear norm, and by $\Vert\cdot\Vert_{\mathbf p}$ the $\mathbf p$-norm (with $1\leq\mathbf p\leq \infty$). Let $\sigma_j(\mathbf X)$ be the $j$th singular value of matrix $\mathbf X$. Denote by $\Vert \cdot\Vert_F$ the Frobenius norm.

\section{Sample complexity of the regularized SAA under low-rankness}\label{main results}

This section presents our main results in Subsection \ref{main results section} after we introduce  our assumptions in Subsection \ref{assumptions here} as well as the definition of the S$^3$ONC in Subsection \ref{S3ONC section}.

\subsection{Assumptions.}\label{assumptions here}
In addition to the low-rankness structure as in Assumption \ref{approximate low-rankness}, we will make the following additional assumptions about continuous differentiability (Assumption \ref{Upper bound}), the tail of the underlying distribution (Assumption \ref{sub exponential condition}), and  a Lipschitz-like continuity (Assumption \ref{Lipschitz condition}).

 \begin{assumption}\label{Upper bound} Let $\mathcal U_L\geq 1$. Assume that  
\begin{align} \left\vert\left.\frac{\partial  f({\mathbf X},\, z)}{\partial \sigma_j({\mathbf X})}\right\vert_{\mathbf X=\mathbf X_1}-\left.\frac{\partial  f({\mathbf X},\, z)}{\partial \sigma_j({\mathbf X})}\right\vert_{\mathbf X=\mathbf X_2}\right\vert<\mathcal U_L\cdot \left\vert \sigma_j(\mathbf X_1)-\sigma_j(\mathbf X_2)\right\vert\label{continuity gradient}
\end{align} for every $j=1,...,p$, all $\mathbf X_1,\,\mathbf X_2\in\mathcal S_p$, and  almost every $z\in \mathcal W$.
\end{assumption}

 
\begin{assumption}\label{sub exponential condition}
The family of random variables, $f({\mathbf X},\, Z_i)-\mathbb E[ f({\mathbf X},\, Z_i)]$, $i=1,...,n$, are independent and follow sub-exponential distributions; that is
$$
\Vert  f({\mathbf X},\, Z_i)-\mathbb E[  f({\mathbf X},\, Z_i)]\Vert_{\psi_1}\leq K,
$$
for some $K\geq 1$ for
 all $\mathbf X\in\mathcal S_p:\,\sigma_{\max}(\mathbf X)\leq R$, where $\Vert\cdot\Vert_{\psi_1}$ is the sub-exponential norm.
\end{assumption} 

Invoking the well-known Bernstein-type inequality, one has that,  for all $\mathbf X\in\mathcal S_p$,  it holds that
\begin{multline}
\mathbb P\left(\left\vert \sum_{i=1}^n a_i\left\{f(\mathbf X,\, Z_i)-\mathbb E[f(\mathbf X,\, Z_i)] \right\}\right\vert>K(\Vert \mathbf a\Vert\sqrt{t}+\Vert \mathbf a\Vert_\infty t)\right)\leq 2\exp\left(-ct\right),
\\\forall t\geq 0, \,\mathbf a=(a_i)\in\Re^n,\label{sub-exponential sequence sum}
\end{multline}
for some absolute constant $c\in(0,\,\frac{1}{2}]$.  \cite[See also][]{Rudelson2013}.

\begin{assumption}\label{Lipschitz condition}
For some measurable and deterministic function $\mathcal C:\, \mathcal W\rightarrow \Re$  with $\mathbb E[\vert \mathcal C(Z)\vert]\leq \mathcal  C_\mu,$ for some $\mathcal  C_\mu\geq 1$, the random variable $\mathcal C(Z)$  satisfies that $\left\Vert \mathcal C(Z)-\mathbb E\left[\mathcal C(Z)\right]\right\Vert_{\psi_1}\leq K_C$ for some $ K_C\geq 1$. Furthermore, $
\vert  f({\mathbf X}_1,\, z)-f({\mathbf X}_2,\, z)\vert\leq \mathcal C(z)\Vert {\mathbf X}_1-{\mathbf X}_2\Vert$
 for all ${\mathbf X}_1,\,{\mathbf X}_2\in\mathcal S_p$, and almost every  $z\in\mathcal W.
$
\end{assumption}
\begin{remark}
Assumption \ref{Upper bound} is easily verifiable and applies to a flexible set of SP problems. 
Assumptions \ref{sub exponential condition} and \ref{Lipschitz condition} are standard, and, by a close examination, it is essentially equivalent to the assumptions made by \cite{Shapiro:book} in the analysis of the traditional SAA.
\end{remark}

\subsection{The significant subspace second-order necessary conditions}\label{S3ONC section}

Our sample complexity results concern critical points that satisfy the S$^3$ONC as per the following definition, where we notice that $P_\lambda(t)$ is twice differentiable for all $t\in(0,\,a\lambda)$.

\begin{definition}\label{SONC definition}
For given $\mathbf Z_1^n\in\mathcal W^n$, a vector $\hat{{\mathbf X}}\in\mathcal S_p$ is said to satisfy the S$^3$ONC  (denoted by {\it S$^3$ONC}$(\mathbf Z_1^n)$) of the problem \eqref{original problem 2} if both of the following  sets of conditions are satisfied:

\begin{enumerate} 
\item[a.] The first-order KKT condition is satisfied at $\mathbf X^{RSAA}$; that is,
\begin{align}
\nabla\mathcal F_{n,\lambda}(\mathbf X^{RSAA},\, \mathbf Z_1^n) =0,\label{first-order kkt}
\end{align}
where $\nabla\mathcal F_{n,\lambda}(\mathbf X^{RSAA},\, \mathbf Z_1^n)$ is the gradient of $\mathcal F_{n,\lambda}(\mathbf X^{RSAA},\, \mathbf Z_1^n)$ at  $\mathbf X^{RSAA}$.
\item[b.] The following inequality holds at $\mathbf X^{RSAA}$ for all $j=1,...,p$:
\begin{align}
 \mathcal U_L+\left[\frac{\partial^2 P_\lambda(\sigma_j(\mathbf X))}{[\partial \sigma_j(\mathbf X)]^2}\right]_{\mathbf X=\mathbf X^{RSAA}}\geq 0,\quad\text{if $\sigma_j(\mathbf X^{RSAA})\in (0,\,a\lambda)$},\label{second condition 1}
\end{align}
where $\mathcal U_L$ is as defined  in  \eqref{continuity gradient} for Assumption \ref{Upper bound}.
 
\end{enumerate}
\end{definition}

As mentioned, the above S$^3$ONC is verifiably a weaker condition than the canonical second-order KKT conditions. Therefore, any local optimization algorithm that guarantees the second-order KKT conditions will necessarily ensure the S$^3$ONC.  

\subsection{Main results}\label{main results section}
Introduce a few short-hand notations:
Denote $\tilde\Delta:=\ln\left(18R\cdot(K_C+\mathcal  C_\mu)\right)$ and $\lambda(\rho):=\sqrt{\frac{8K(2p+1)^{2/3}s^{-\rho}}{c\cdot a\cdot n^{2/3}}[\ln(n^{1/3}p)+\tilde\Delta]}$, for the same $c$ in \eqref{sub-exponential sequence sum} and a user-specific $\rho\geq 0$.  
Recall the definition of $ {\mathbf X}^{\ell_1}_{\lambda}$  in \eqref{nuclear problem} and
specify  $a^{-1}=2\mathcal U_L$ (and thus $a<{\mathcal U_L}^{-1}$.  We are now ready to present our claimed results.

.
\begin{theorem}\label{first theorem}  Suppose that Assumptions 1 through 4 hold.    Specify the penalty parameter $\lambda:=\lambda(\rho)$. Let $\mathbf X^{RSAA}\in\mathcal S_p:\sigma_{\max}(\mathbf X^{RSAA})\leq R$ satisfy the S$^3$ONC$(\mathbf Z_1^n)$   to \eqref{original problem 2}     almost surely.  For any $\Gamma\geq 0$ and some    universal constants $\tilde c,\,C_1>0$, if
\begin{align}
n > C_1\cdot s^{3\rho}\cdot \left[\left(\frac{\Gamma}{K}\right)^{3}+1\right]\cdot p+C_1\cdot s\cdot p\cdot\left(\ln(n^{1/3}p)+\tilde \Delta\right),\label{sample initial requirement 2 3}
\end{align}
 and   $\mathcal F_{n,\lambda}(\mathbf X^{RSAA},\,\mathbf Z_1^n)\leq \mathcal F_{n,\lambda}({\mathbf X}^{*},\,\mathbf Z_1^n)+\Gamma$ almost surely, then the  excess risk is bounded by 
\begin{multline}
\mathbb F(\mathbf X^{RSAA})-\mathbb  F({\mathbf X^*})\leq \sqrt{\frac{K\cdot s^\rho\cdot p^{1/3}\cdot \Gamma}{n^{1/3}}} +\Gamma
\\ +C_1K\cdot\left[\frac{s^{1-\rho}\cdot p^{2/3}\cdot \left(\ln(n^{1/3}p)+\tilde \Delta\right)}{n^{2/3}}+\sqrt{\frac{s \cdot p\cdot  \left(\ln(n^{1/3}p)+\tilde \Delta\right)}{n}}+\frac{p^{1/3}\cdot s^\rho}{n^{1/3}}\right],\label{bound global 3}
\end{multline}
with probability at least 
$
1-2 (p+1) \exp(-\tilde cn)-6\exp\left(-2c(2p+1)^{2/3}n^{1/3}\right)
$.
\end{theorem}
\begin{proof}
See proof in Section \ref{proof of the desired result}.
\end{proof}

\begin{remark}
 Some   explanations on the notations are below:
\begin{enumerate}[leftmargin=12pt,labelsep=5pt]
\item $\Gamma$ measures the solution quality in solving the (in-sample) RSAA formulation; that is, $\Gamma$ is the suboptimality gap of minimizing the RSAA, which is the surrogate model for the true SP problem in \eqref{original problem}. We refer to $\Gamma$   as ``in-sample suboptimality gap'' hereafter. 
\item More important to us is a second type of suboptimality gap, which we refer to as the ``out-of-sample suboptimality gap'', calculated as $\mathbb F(\mathbf X)-\mathbb  F({\mathbf X^*})$ for a feasible solution $\mathbf X$. The out-of-sample suboptimality gap measures how well the solution $\mathbf X$ optimizes the true SP problem in \eqref{original problem}.
\item $\tilde \Delta$ is some logarithmic terms independent of $p$ and $n$.
\item $K$ and $K_C$ are subexponential norms of the underlying distributions. They are alternative measures of the distributions' variances. 
\end{enumerate}
\end{remark}
\begin{remark}
Some intuitions on the above theorem are as follows:
\begin{enumerate}[leftmargin=12pt,labelsep=5pt]
\item Theorem \ref{first theorem} ensures that all S$^3$ONC solutions to the RSAA formulation yield a bounded out-of-sample suboptimality gap in minimizing the true problem \eqref{original problem}.  
\item Furthermore, the out-of-sample suboptimality gap is consistent with the in-sample suboptimality gap in the sense that  the former deteriorates as $\Gamma$ increases.  When $\Gamma$ is relatively large, the deterioration is dominated by a linear rate. 
\end{enumerate}
\end{remark}

We may well control the in-sample suboptimality gap $\Gamma$ by properly initializing the search for an S$^3$ONC solution. Indeed, as is shown in the corollary below, using $ {\mathbf X}^{\ell_1}_{\lambda}$  defined in \eqref{nuclear problem} to warm-start any S$^3$ONC-guaranteeing local optimization algorithm ensures the promised sample complexity.
\begin{corollary}\label{first corollary}
   Suppose that Assumptions 1 through 4 hold.    Specify the penalty parameter  $\lambda=\lambda(0)$ (that is, $\rho=0$) in both formulations \eqref{nuclear problem} and \eqref{original problem 2}. Let $\mathbf X^{RSAA}\in\mathcal S_p:\sigma_{\max}(\mathbf X^{RSAA})\leq R$ satisfy the S$^3$ONC$(\mathbf Z_1^n)$   to \eqref{original problem 2}     almost surely.   For  some    universal constant $\tilde c,\,C_2>0$, if
\begin{align}
n >  C_2\cdot p\cdot \mathcal U_L\cdot [\ln(n^{\frac{1}{3}}p)+\tilde \Delta]\cdot s^{ \frac{3}{2} }R^{\frac{3}{2}},\label{sample initial requirement 2} 
\end{align}
 and   \begin{align}\mathcal F_{n,\lambda}(\mathbf X^{RSAA},\,\mathbf Z_1^n)\leq \mathcal F_{n,\lambda}( {\mathbf X}^{\ell_1}_{\lambda},\,\mathbf Z_1^n)\label{initializer}
 \end{align}
  almost surely, where $ {\mathbf X}^{\ell_1}_{\lambda}$ is as defined in \eqref{nuclear problem}, then the excess risk is bounded by 
\begin{multline}
\mathbb F(\mathbf X^{RSAA})-\mathbb  F({\mathbf X^*}) 
\\\leq C_2\cdot s\cdot K\cdot\left[\frac{p^{2/3} \left(\ln(n^{\frac{1}{3}}p)+\tilde \Delta\right)}{n^{\frac{2}{3}}}+\frac{p^{1/3}R\cdot \mathcal U_L^{1/2}\sqrt{\ln(n^{\frac{1}{3}}p)+\tilde \Delta}}{n^{\frac{1}{3}}}\right],\label{bound 123}
\end{multline}
with probability at least 
$1-2 (p+1) \exp(-\tilde cn)-6\exp\left(-2c(2p+1)^{2/3}n^{1/3}\right)$.
\end{corollary}
\begin{proof}
See proof in Section \ref{proof of the desired result corollary 1}.
\end{proof}


\begin{remark}
We would like to make a few remarks on the above result:
\begin{enumerate}[leftmargin=12pt,labelsep=5pt]
\item Corollary \ref{first corollary} above  establishes our claimed result of almost linear complexity at an S$^3$ONC solution generated with a proper initialization.

\item
The same corollary considers the particular sublevel set that has a better objective value (in terms of RSAA formulation) than $ {\mathbf X}^{\ell_1}_{\lambda}$. In such a case, the suboptimality in minimizing the true problem \eqref{original problem} explicitly vanishes as sample size $n$ increases.  
\item  $ {\mathbf X}^{\ell_1}_{\lambda}$ is an initial  solution often tractably computable under the common assumption that $f(\,\cdot\,,z)$ is convex for almost every $z\in\mathcal W$.   However, our results in Theorem \ref{first theorem} is not contingent on the convexity of $f(\,\cdot\,,z)$, although generating $ {\mathbf X}^{\ell_1}_{\lambda}$ may be intractable when convexity of $f(\,\cdot\,,z)$ is not in presence.

\item 
Corollary \ref{first corollary} above is consistent with the claimed sample complexity in \eqref{new log bound}, which is almost linear in $p$. Indeed, for achieving an accuracy of $\epsilon$, the above bounds indicate a sample complexity $\tilde O(1)\cdot \frac{p}{\epsilon^3}\cdot \text{\it polylog}(p,\,\frac{1}{\epsilon})$, which is almost linear in $p$, for some quantities $\tilde O(1)$ that is independent of $n$, $\epsilon$, and $p$. 
 \end{enumerate}
\end{remark}

We note that the dependence of sample size $n$ on rank $s$ of the true solution $\mathbf X^*$ is cubic, which means that the proposed approach is more powerful when the true solution $\mathbf X^*$ is of very low rank. The deterioration may be fast when $s$ increases. Nonetheless, we believe it possible to significantly reduce the order on $s$ if any further information below is given: (i) If the $\mathcal F_n$ or $\mathbb F$ satisfies strong convexity or its certain relaxed forms, dependence on $s$ is likely reducible, as it has been successful for \cite{Liu2018} in stochastic optimization under sparsity. (ii) If the value of $s$ can be coarsely predicted in the sense that $O(1)\cdot s$ for some universal constant $O(1)$ is given, then one may also properly modify the value of $\lambda$ to decrease the dependence on $s$. We will consider the relatively special case in (i) in future study. Nonetheless, our claim in (ii) above is provided in Corollary \ref{second corollary} below.

\begin{corollary}\label{second corollary}
Suppose that Assumptions 1 through 4 hold.    Specify the penalty parameter  $\lambda=\lambda(\frac{2}{3})$ (that is, $\rho=\frac{2}{3}$) in both formulations \eqref{nuclear problem} and \eqref{original problem 2}. Let $\mathbf X^{RSAA}\in\mathcal S_p:\sigma_{\max}(\mathbf X^{RSAA})\leq R$ satisfy the S$^3$ONC$(\mathbf Z_1^n)$   to \eqref{original problem 2}     almost surely.  For  some    universal constant $\tilde c,\,C_3>0$, if
\begin{align}
n >  C_3\cdot p\cdot \mathcal U_L\cdot [\ln(n^{\frac{1}{3}}p)+\tilde \Delta]\cdot s^{ 2 }\cdot R^{\frac{3}{2}},\label{sample initial requirement 2 new 234} 
\end{align}
 and   and \eqref{initializer} holds  almost surely, where $ {\mathbf X}^{\ell_1}_{\lambda}$ is as defined in \eqref{nuclear problem}, then the excess risk is bounded by 
\begin{multline}
\mathbb F(\mathbf X^{RSAA})-\mathbb  F({\mathbf X^*}) 
\\\leq C_3\cdot   K\cdot\left[\frac{s^{1/3}p^{2/3} \left(\ln(n^{\frac{1}{3}}p)+\tilde \Delta\right)}{n^{\frac{2}{3}}}+\frac{s^{2/3}p^{1/3}\cdot R\cdot \mathcal U_L^{1/2}\cdot \sqrt{\ln(n^{\frac{1}{3}}p)+\tilde \Delta}}{n^{\frac{1}{3}}}\right],\label{bound 12345}
\end{multline}
with probability at least 
$1-2 (p+1) \exp(-\tilde cn)-6\exp\left(-2c(2p+1)^{2/3}n^{1/3}\right)$.

\end{corollary}
\begin{proof}
See proof in Section \ref{proof of the desired result corollary 2}.
\end{proof}

\begin{remark}
The Corollary \ref{second corollary}, similar to Corollary \ref{first corollary},   establishes our claimed result of almost linear complexity at a computable S$^3$ONC solution generated with a proper initialization, $ {\mathbf X}^{\ell_1}_{\lambda}$, which can be tractable when $f(\,\cdot\,,z)$ is convex for almost every $z$. \end{remark}

\begin{remark}
In contrast to Corollary \ref{first corollary}, Corollary \ref{second corollary} yields a sample complexity with much reduced dependence on $s$; quadratic instead of cubic in $s$. We suppose that this dependence is no longer improvable. This is because, even if we are given the exact knowledge to correctly reduce the ``redundant'' dimensions of the problem, the  traditional SAA to the reduced problem will still require a sample size quadratically dependent on $s$.
\end{remark}

\begin{remark}
There is strong correspondence between the SP and statistical learning as formerly noted by \cite{Liu2018,MandB2011}. More specifically, the SAA formulation \eqref{original problem 2} can be considered as  an M-estimation problem and the suboptimality gap $\mathbb F(\mathbf X^{RSAA})-\mathbb F({{\mathbf X}^*})$ has the same formulation as the excess risk discussed by  \cite{excess risk}, \cite{Koltchinskii 2010}, and \cite{Clemencon}. We therefore argue that the results in Theorem \eqref{first theorem} and Corollaries \ref{first corollary} and \ref{second corollary} indicate that M-estimation with  high dimensions is generally possible under a low-rankness assumption. In particular, since our analysis does not assume any form of RSC, we believe that our results then provides perhaps the first out-of-sample performance guarantee for high-dimensional low-rank estimation beyond RSC.
\end{remark}

\begin{remark}
We would like to remark again that, to obtain the desired results, the incurred computational ramification can be reasonably small. This is because $\mathbf X^{RSAA}$ is only a stationary point that satisfies \eqref{initializer}. First, the stationarity can be ensure by invoking local optimization algorithms. Second, the stipulated inequality in \eqref{initializer} can be ensured by initializing the local algorithm with $ {\mathbf X}^{\ell_1}_{\lambda}$. Such an initializer often can be generated within polynomial time  under the common assumption that $f(\,\cdot\,,w)$ is convex for almost every $w\in\mathcal W$, although the convexity of $f(\,\cdot\,,w)$ is not necessary for proving the claimed almost linear sample complexity.
\end{remark}

\section{Proof Overview and Techniques}\label{General proof}
\subsection{General ideas}
The general idea of our proof is straightforward and focuses on addressing the question: {\it how to show that an S$^3$ONC solution has low rank.} If this question is answered, then the desired results can be almost evident by analyzing the  $\epsilon$-net   for all  the low-rank subspaces. Such an analysis is available in Lemma 3.1 of \cite{Candes2011} and is restated (with minor modifications) in Lemma \ref{low rank epsilon net} herein.

To that end, we utilize a unique property of the MCP function, which ensure that the stationary points that satisfy the S$^3$ONC solutions $\mathbf X^{RSAA}$ must obey a thresholding rule: for all the singular values, they must be either 0 or greater than $a\lambda$. This means that for each nonzero singular value in the S$^3$ONC solution $\mathbf X^{RSAA}$, an additional penalty of value $\frac{a\lambda^2}{2}$ is added to the objective function of the RSAA, and, therefore, the total penalty incurred by the low-rankness-inducing regularization is $\sum_{j=1}^pP_\lambda(\sigma_j(\mathbf X^{RSAA}))={{\bf rk}(\mathbf X^{RSAA})}\cdot\frac{a\lambda^2}{2}$. Now, consider those stationary points whose suboptimality gaps (in minimizing the RSAA)  are smaller than a user-specific quantity $\Gamma$, and therefore,  $\mathcal F_{n,\lambda}(\mathbf X^{RSAA},\mathbf Z_1^n)=\mathcal F_{n}(\mathbf X^{RSAA},\mathbf Z_1^n)+{{\bf rk}(\mathbf X^{RSAA})}\cdot\frac{a\lambda^2}{2}\leq \mathcal F_{n,\lambda}(\mathbf X^{*},\mathbf Z_1^n)+\Gamma$. The rank of such $\mathbf X^{RSAA}$ rank must be bounded from above by a function of $\Gamma$. Such a function can be explicated via a peeling technique discussed by \citep{Raskutti}. Some relative details are provided below.

\subsection{Proof Roadmap}
The proof of Theorem \ref{first theorem} is motivated by \cite{Liu2019} but  involves substantial generalization from an SP problem under sparsity in  \cite{Liu2019}  to an SP problem under low-rankness herein. To understand the non-trivial step involved in this generalization,  one may observe the fundamental differences between those two problems: While low-rankness can be represented by sparsity via a linear transformation,  the linear operator involved in this transformation is completely unknown. More specifically, by singular value decomposition, one may write $\mathbf X^*:=UD^*V^\top$ for some proper unitary matrices $U$ and $V$. Apparently, as per Assumption \ref{approximate low-rankness}, the diagonal matrix $D^*$ must be sparse and  $U^\top$ and $V$ are linear operators that project $\mathbf X^*$ to  a sparse domain; indeed, $D^*=U^\top \mathbf X^*V$.   Nonetheless, the knowledge on $U^\top$ and $V$  are completely absent, which leads to significant ramifications in analysis. 

The following are general explanations on the key steps, where $\tilde O(1)$'s denote (potentially different) quantities that are independent of $p$ and $n$:
\begin{description}[style=unboxed,leftmargin=0cm]
\item[Step 1. {\it The thresholding rule of the MCP.}] Under the assumption that  $\mathcal U_L<a^{-1}$, in Proposition \ref{Proposition 1}, we show that, for an S$^3$ONC solution to the RSAA formulation, denoted $\mathbf X^{RSAA}$,  a thresholding rule of $\sigma_j(\mathbf X^{RSAA})$, for all $j$, is that $\sigma_j(\mathbf X^{RSAA})\neq 0\Longrightarrow \sigma_j(\mathbf X^{RSAA})\geq a\lambda$, where $a$ and $\lambda$ are the tuning parameters of the MCP function, $P_\lambda$. This can be demonstrated by observing that  the definition of the S$^3$ONC, which is $\mathcal U_L-P''_\lambda(\sigma_j(\mathbf X^{RSAA}))=\mathcal U_L-\frac{1}{a}\geq 0$ if $\sigma_j(\mathbf X^{RSAA})\in(0,\,a\lambda)$,   contradicts with the assumption that $\mathcal U_L<a^{-1}$. Therefore, it holds that $\sigma_j(\mathbf X^{RSAA})\geq a\lambda$, unless $\sigma_j(\mathbf X^{RSAA})=0$.
\item[Step 2. {\it $\epsilon$-net argument for  low-rank subspaces.}]  We apply the well-known $\epsilon$-net argument to show a point-wise error bound for $\vert\mathcal F_{n,\lambda}(\mathbf X,\mathbf Z_1^n)-\mathbb F(\mathbf X)\vert\leq \epsilon$  for all $\mathbf X\in\mathcal S_p: \sigma_{\max}(\mathbf X)\leq R$ in all rank-$\tilde p$ subspaces, whose elements have rank no greater than a given  $\tilde p$. To that end,  first observe that, for any rank-$\tilde p$ subspace, the standard $\epsilon$-net argument results in a covering number of $\tilde O(1)\left(  \sqrt{\tilde p}\cdot\frac{\tilde O(1)}{\epsilon}\right)^{(2p+1)\tilde p}$. Second, since there can be   ${p}\choose{\tilde p}$-many rank-$\tilde p$ subspaces, the total covering number for all possible rank-$\tilde p$ subspaces is $${{p}\choose{\tilde p}}\cdot \left(  \sqrt{\tilde p}\cdot\frac{\tilde O(1)}{\epsilon}\right)^{\tilde p}\leq \left(  \tilde O(1)\cdot\frac{p}{\epsilon}\right)^{(2p+1)\tilde p}.$$ Combining this covering number, the Bernstein-like inequality, and Lipschitz-like inequality in \eqref{sub-exponential sequence sum}, we have  that, for any $t\geq 0$,
\begin{multline}\vert\mathcal F_{n}(\mathbf X,\mathbf Z_1^n)-\mathbb F(\mathbf X)\vert
> \tilde O(1)\cdot\frac{t}{n} +\tilde O(1)\cdot\sqrt{\frac{t}{n}} + \epsilon,
\\~\text{$\forall\,\mathbf X\in\mathcal S_p:\,\sigma_{\max}(\mathbf X)\leq R:\,{\bf rk}(\mathbf X)\leq \tilde p$},\label{to show bound prob}
\end{multline} with probability at most $\left( \tilde O(1)\cdot \frac{p}{\epsilon}\right)^{(2p+1)\tilde p}\exp(-ct)+ \exp(-\tilde O(1)\cdot n)$ for some universal constant $c\in(0,1/2]$. We may choose   to let $t=2 \tilde p(2p+1)\ln\left(\frac{\tilde O(1)\cdot p}{\epsilon}\right)$, as well as $\epsilon=n^{-\frac{1}{3}}$, then, observe that the probability the fact (we will call it {\bf Observation ($\star$)}, to be useful later in Step 4) that the first term in the probability is vanishing exponentially fast to zero as $\tilde p$ increases and the second term is independent of $\tilde p$.
\item[Step 3. {\it An implication of Step 2.}] Let $\mathbf X^{RSAA}$  be an S$^3$ONC solution to the RSAA formulation in \eqref{original problem 2}. Assume that $\mathbf X^{RSAA}$ is within the $\Gamma$-sublevel set for some $\Gamma\geq 0$. Then, (cf. Assumption \ref{approximate low-rankness}) it is straightforward to obtain  from the fact that $0\leq P_\lambda(\,\cdot\,)\leq \frac{a\lambda^2}{2}$ and the results of Step 1 (i.e., $ \sigma_j(\mathbf X)\geq a\lambda$, unless $\sigma_j(\mathbf X)=0$),
\begin{align}\mathcal F_{n}(\mathbf X^{RSAA},\mathbf Z_1^n) +{\bf rk}(\mathbf X^{RSAA})\cdot \frac{a\lambda^2}{2}\leq  \mathcal F_{n}(\mathbf X^*,\mathbf Z_1^n)+\frac{a\lambda^2\cdot s}{2}+\Gamma.\label{test and obtain new result}
\end{align}
  If ${\bf rk}(\mathbf X^{RSAA})\leq \tilde p$, the result from Step 2 can be invoked to bound the differences, $\mathbb F(\mathbf X^{RSAA}) -\mathcal F_{n}(\mathbf X^{RSAA},\mathbf Z_1^n) $ and $\mathcal F_{n}(\mathbf X^{*},\mathbf Z_1^n)-\mathbb F(\mathbf X^{*})$, to be smaller than a desired level. In particular, as we choose to let $t=2 \tilde p(2p+1)\ln\left(\frac{\tilde O(1)\cdot p}{\epsilon}\right)$, as well as $\epsilon=n^{-\frac{1}{3}}$, in \eqref{to show bound prob} and $\lambda = \tilde{O}(1)\cdot\frac{p^{1/3}\sqrt{\ln(np)}}{n^{1/3}}$ in \eqref{test and obtain new result}. After some algebraic simplification, we  obtain that
\begin{align}
&\mathbb F({\mathbf X}^{RSAA})-\mathbb F(\mathbf X^*)\nonumber
\\\leq&\, -\frac{a\lambda^2}{2}{\bf rk}(\mathbf X^{RSAA}) +\tilde O(1)\cdot \frac{sp^{2/3}\ln (pn)}{n^{2/3}}+\tilde O(1)\cdot \sqrt{\frac{\tilde p}{n}\ln\left(pn\right)}+\frac{p^{1/3}}{n^{1/3}}+\Gamma\label{to use later for explanation}
\\\leq&\, \tilde O(1)\cdot \frac{\tilde p\cdot  p\ln (pn)}{n}+\tilde O(1)\cdot \sqrt{\frac{\tilde p\cdot p}{n}\ln\left(pn\right)}+\frac{1}{n^{1/3}}+\Gamma\label{useful later on}
\end{align}
with probability at least $1-\exp\left(-\tilde O(1)\cdot\tilde p\cdot p\cdot \ln\left(np\right)\right)-\exp(-\tilde O(1)\cdot n)$. Recalling that $\tilde p$ is an upper bound on the rank of $\mathbf X^{RSAA}$, the above result in \eqref{useful later on} is now close to the desired ``almost linear'' sample complexity results if $\tilde p$ much smaller than $p$. As it turns out, it is indeed the case. As is demonstrated in Theorem \ref{first theorem}, we can show that ${\bf rk}(\mathbf X^{RSAA})\leq   \tilde p:=\tilde O(1)\cdot \left(s+\frac{n^{1/3}}{p^{1/3}}+\frac{n^{1/3}}{p^{1/3}}\cdot \Gamma\right)$, which is to be explained subsequently. 
\item[Step 4. {\it Upper bound on ${\bf rk}(\mathbf X^{RSAA})$.}] From Step 3, we observe that the desired result in Theorem \ref{first theorem} can be shown by proving that 
\begin{align}{\bf rk}(\mathbf X^{RSAA})\leq \tilde O(1)\cdot \left(s+\frac{n^{1/3}}{p^{1/3}}+\frac{n^{1/3}}{p^{1/3}}\cdot \Gamma\right).\label{relationship here needed}
\end{align} To that end, we may invoke a scheme motivated by the peeling technique discussed by \citep{Raskutti}. We will show in Proposition \ref{bound dimensions} that, for some integer $\tilde p_u:= \tilde O(1)\cdot \left(s+\frac{n^{1/3}}{p^{1/3}}+\frac{n^{1/3}}{p^{1/3}}\cdot \Gamma\right)$,  it holds that, for all  $\tilde p\geq \tilde p_u$, the inequality in \eqref{to use later for explanation} cannot be satisfied given $\{{\bf rk}(\mathbf X^{RSAA})\geq \tilde p\}$; this is because the first (negative) term therein would have too large a magnitude and render the whole composite  on the right-hand-side of \eqref{to use later for explanation} a negative quantity, which implies $\mathbb F({\mathbf X}^{RSAA})-\mathbb F(\mathbf X^*)<0$ and contradicts with the fact that $\mathbf X^*$ minimizes $\mathbb F$ by definition. Since $\{\text{\eqref{to use later for explanation} holds}\}\cap \{{\bf rk}(\mathbf X^{RSAA})\geq \tilde p\}\supseteq\{{\bf rk}(\mathbf X)=\tilde p\}\cap\{\text{The complement to \eqref{to show bound prob}  holds with given $\tilde p$}\}$, it then implies that, for all $\tilde p\geq \tilde p_u$,
\begin{align}
0=\mathbb P\left[\{{\bf rk}(\mathbf X^{RSAA})=\tilde p\}\cap \{\text{The complement to \eqref{to show bound prob}  holds with given $\tilde p$}\}\right].\nonumber
\end{align} 
As an immediate result,   $\mathbb P[{\bf rk}(\mathbf X)=\tilde p]\leq \mathbb P[\{\text{\eqref{to show bound prob} holds with given $\tilde p$}\}]$ for all $\tilde p:\,\tilde p\geq \tilde p_u$. Therefore, invoking union bound and De Morgan's law, $\mathbb P[{\bf rk}(\mathbf X)\leq \tilde p_{u}-1]\geq 1-\sum_{\tilde p= \tilde p_u}^p \mathbb P[{\bf rk}(\mathbf X)=\tilde p]\geq 1-\sum_{\tilde p= \tilde p_u}^p \mathbb P[\{\text{\eqref{to show bound prob} holds with given $\tilde p$}\}]$.   By our choice of parameters for $t$ and $\epsilon$ as in Step 2, the {\bf Observation ($\star$)} (which is defined in Step 2) leads to a simplification of the probability bound by noting $\sum_{\tilde p= \tilde p_u}^p \mathbb P[\{\text{\eqref{to show bound prob} holds with given $\tilde p$}\}]$ involves the sum of a geometric sequence plus a term vanishing exponentially in $n$.  Combining the results from Step 4 with Step 3, we can then show   Theorem \ref{first theorem} after some algebraic simplification.


\item[Step 5. {\it To show Corollaries \ref{first corollary} and \ref{second corollary}.}] Both corollaries can be shown by noticing that $ {\mathbf X}^{\ell_1}_{\lambda}$ yields a suboptimality gap of no more than $\tilde{O}(1)\cdot \lambda\cdot s\cdot R$ when we choose $\lambda=\tilde{O}(1)\cdot\frac{p^{1/3}\sqrt{\ln(np)}}{n^{1/3}s^{-\rho/2}}$ in \eqref{nuclear problem} (which share the same $\lambda$ value as  in \eqref{original problem 2}). Specifically, Corollary \ref{first corollary} is shown with $\rho=0$ and Corollary \ref{second corollary} is shown with $\rho= 2/3$.
\end{description}
\section{Conclusions}\label{conclusions}
This paper proposes a regularized SAA (RSAA), which is incorporates a low-rankness-exploiting regularization into the traditional SAA framework, to solve  high-dimensional SP problems of minimizing an expected function over a $p$-by-$p$ matrix argument. We prove that  certain stationary points   ensure an almost linear sample complexity:  the RSAA only requires a sample size  almost linear in $p$ to achieve sound optimization quality, while, in contrast, the required sample size for the traditional SAA is at least quadratic in  $p$. The reduced sample complexity can be obtained at certain stationary points without incurring a significant computational effort, especially when the cost function $f(\,\cdot\,,z)$ is convex for almost every $z\in\mathcal W$. Our RSAA theory also implies that, under the low-rankness assumption,  high-dimensional matrix estimation   is generally possible beyond linear and generalized linear models even if $p$, the size of the matrix  to be estimated, is large and the RSC is absent. Future research will focus on generalizing our paradigm to problems with general linear and nonlinear constraints. Furthermore, we will investigate the (non-)tightness of our bound on sample complexity.

%
\begin{appendix}
\section{Technical proofs}\label{proof section}

\subsection{Proof of Theorem \ref{first theorem}}\label{proof of the desired result}
The proof follows the argument of  Proposition 1 in \cite{Liu2019} and makes important generalizations  from handling sparsity to low-rankness. Furthermore, much more flexible choices of penalty parameters $\lambda$ is enabled. We follow the same set of notations in Proposition \ref{core proposition 1} in defining $\tilde p_u$, $\epsilon$, and
$
\Delta_1(\epsilon):=\ln\left(\frac{18  p R\cdot(K_C+\mathcal C_\mu)}{\epsilon}\right)$.  Furthermore, we will let $\epsilon:={\frac{1}{n^{1/3}}}$ and   $\tilde \Delta:=\ln\left(18 \cdot R\cdot(K_C+\mathcal  C_\mu)\right)$. Then  $\Delta_1(\epsilon)=\ln\left(\frac{18\cdot  (K_C+\mathcal  C_\mu)\cdot p\cdot R}{\epsilon}\right)=\ln(n^{1/3}p)+\tilde \Delta>0$ and $\lambda=\sqrt{\frac{8\cdot s^{-\rho}\cdot K(2p+1)^{2/3}\cdot  \Delta_1(\epsilon)}{c\cdot a\cdot n^{2/3}}}=\sqrt{\frac{8\cdot s^{-\rho}\cdot K \cdot (2p+1)^{2/3}}{c\cdot a\cdot n^{2/3}}[\ln(n^{1/3} p)+\tilde \Delta]}$.  We will denote by $O(1)$'s  universal constants, which may be different in each of their occurence.

To show the desired results, it suffices to simplify the results in Proposition \ref{core proposition 1}.  
We will first derive an explicit form for $\tilde p_u$. To that end, we let ${P_X}:= { \tilde p_u}$ and $T_1:=2P_\lambda(a\lambda)-\frac{8K\cdot(2p+1) }{cn}\Delta_1(\epsilon)$. We then solve the following inequality, which is equivalent to \eqref{condition on value lambda} of Proposition \ref{core proposition 1}, for  a feasible $P_X$,
 \begin{align}
\frac{T_1}{2}\cdot P_X-\frac{2 K  }{\sqrt{n}}\sqrt{\frac{2P_X\cdot (2p+1)\Delta_1(\epsilon)}{c}}>\Gamma+2\epsilon+sP_\lambda(a\lambda),\label{initial inequality}
\end{align}
for   the same $c\in(0,\, 0.5]$ in \eqref{sub-exponential sequence sum}.
Solving the above inequality in terms of $P_X$, we have
$
\sqrt{P_X}> \frac{2 K}{T_1\sqrt{n}}\sqrt{\frac{2(2p+1)\cdot\Delta_1(\epsilon)}{c}}+
\frac{\sqrt{\frac{2(2 K )^2\cdot(2p+1)\cdot\Delta_1(\epsilon)}{cn}+2T_1[\Gamma+2\epsilon+sP_\lambda(a\lambda)]}}{T_1}.$
To find a feasible $P_X$, we may as well let
$
{P_X}> \frac{32K^2\cdot (2p+1)\cdot\Delta_1(\epsilon)}{cT_1^2\cdot n}+8T_1^{-1}[\Gamma+2\epsilon+sP_\lambda(a\lambda)].$
For  $\lambda=\sqrt{\frac{8K\cdot s^{-\rho} \cdot \Delta_1(\epsilon)\cdot (2p+1)^{2/3}}{c\cdot a\cdot n^{2/3}}}=\sqrt{\frac{8K\cdot s^{-\rho}\cdot (2p+1)^{2/3}}{c\cdot a\cdot n^{2/3}}[\ln(n^{1/3}p)+\tilde \Delta]}$ with $\tilde \Delta:=\ln\left( 18 \cdot R\cdot(K_C+\mathcal  C_\mu)\right)$),  we have $P_\lambda(a\lambda)=\frac{a\lambda^2}{2}=\frac{4K \cdot s^{-\rho}\cdot(2p+1)^{2/3}}{c\cdot n^{2/3}}\cdot \Delta_1(\epsilon)$. Furthermore, $2P_\lambda(a\lambda)=\frac{8K \cdot s^{-\rho}\cdot (2p+1)^{2/3}\cdot \Delta_1(\epsilon)}{c\cdot n^{2/3}}> \frac{4 \cdot s^{-\rho} K\cdot  \Delta_1(\epsilon)\cdot (2p+1)^{2/3}}{c\cdot n^{2/3}}+\frac{8K \cdot (2p+1)}{nc}\Delta_1(\epsilon)$ as per our assumption (i.e., \eqref{sample initial requirement 2 3} implies that $n^{1/3}> 2 s^{\rho}$). Therefore, $T_1=2P_\lambda(a\lambda)- \frac{8K\cdot(2p+1)}{nc}\Delta_1(\epsilon)>\frac{4 K \cdot s^{-\rho}\cdot \Delta_1(\epsilon)\cdot (2p+1)^{2/3}}{c\cdot n^{2/3}}$. Hence, if we recall $\epsilon=n^{-1/3}$, to satisfy \eqref{initial inequality}, it suffices to let $P_X$ be any integer that satisfies 
$
P_X
\geq \frac{2cn^{1/3} s^{2\rho}}{\Delta_1(n^{-\frac{1}{3}})\cdot (2p+1)^{2/3}\cdot }+\frac{2cn^{2/3}s^{\rho}}{K \Delta_1(n^{-\frac{1}{3}})\cdot (2p+1)^{2/3}\cdot }\cdot\left[\Gamma+\frac{2}{n^{1/3}}+sP_\lambda(a\lambda)\right],$
which is satisfied by letting $P_X\geq \tilde p_u$ with
\begin{align}
\tilde p_u
:=& \left\lceil\frac{2cn^{1/3}s^{2\rho}}{\Delta_1(n^{-\frac{1}{3}})\cdot (2p+1)^{1/3}}+\frac{2cn^{2/3}s^{\rho}}{K\cdot \Delta_1(n^{-\frac{1}{3}})\cdot (2p+1)^{2/3}}\cdot\left(\Gamma+\frac{2}{n^{1/3}}\right)+8s\right\rceil.\label{assignment of pu new}
\end{align}
In the meantime, verifiably, $\tilde p_u>s$.
 Since the above is a sufficient to ensure \eqref{initial inequality}, we know that  \eqref{condition on value lambda} in Proposition \ref{core proposition 1} holds for any $\tilde p:\, \tilde p_u\leq \tilde p\leq p$.
Due to Proposition \ref{core proposition 1},   with probability at least
$
P^*:= \, 1-6\exp\left(-\tilde p_u\cdot (2p+1)\cdot \Delta_1(n^{-\frac{1}{3}})\right) -2 (p+1) \exp(-\tilde cn) \nonumber
\geq 1-6\exp(-2c\cdot(2p+1)^{2/3} \cdot n^{1/3})-2 (p+1) \exp(-\tilde cn)$, it holds that 
\begin{multline}
\mathbb F(\mathbf X^{RSAA})-\mathbb  F({\mathbf X}^*)
\leq s\cdot P_\lambda(a\lambda)+ {\frac{2K}{\sqrt{n}}}\sqrt{\frac{2 \tilde p_u(2p+1)}{c}\Delta_1(n^{-\frac{1}{3}})} 
\\+\frac{4K }{n}\frac{ \tilde p_u(2p+1)}{c}\Delta_1(n^{-\frac{1}{3}})+2\epsilon+\Gamma, \label{first results obtained}
\end{multline}
in which $\tilde p_u$ is as per \eqref{assignment of pu new}.

The following simplifies the formula while seeking to preserve the rates in  $n$ and $p$. Firstly, we have \begin{align}&\sqrt{\frac{2 \tilde p_u\cdot (2p+1)}{cn}\Delta_1(n^{-\frac{1}{3}})}  
\\ \leq&\,\sqrt{\frac{4\cdot (2p+1)s^{2\rho}}{cn\cdot (2p+1)^{1/3}}\Delta_1(n^{-\frac{1}{3}}) \cdot\frac{cn^{1/3}}{\Delta_1(n^{-\frac{1}{3}})}+\frac{4cn^{2/3}(2p+1)s^{\rho}}{K(2p+1)^{2/3} \Delta_1(n^{-\frac{1}{3}})}\left(\Gamma+\frac{2}{n^{1/3}}\right)\cdot \frac{\Delta_1(n^{-\frac{1}{3}})}{cn}}  \nonumber
\\&+\sqrt{\frac{2}{cn}\Delta_1(n^{-\frac{1}{3}})\cdot \left( 8s+1\right)\cdot(2p+1)} \nonumber
\\
\leq&\,\sqrt{\frac{4(2p+1)^{2/3}s^{2\rho}}{n^{2/3}}+\frac{4s^{\rho}\cdot (\Gamma+\frac{2}{n^{1/3}})\cdot (2p+1)^{1/3}}{K n^{1/3}}}+  \sqrt{\frac{2}{nc}\Delta_1(n^{-\frac{1}{3}})\cdot \left( 8s+1\right)\cdot(2p+1)},\label{first bound here}
\end{align}
which is due to $\sqrt{x+y}\leq \sqrt{x}+\sqrt{y}$ for any $x,\, y\geq 0$ and the relations that $0<a<\mathcal U_L^{-1}\leq 1$, $0<c\leq 0.5$, $ K\geq 1$, and $\Delta_1(n^{-\frac{1}{3}})\geq\ln 36$.

Similar to the above, we obtain 
\begin{multline}
  {\frac{3 \tilde p_u\cdot (2p+1)}{cn}\Delta_1(n^{-\frac{1}{3}})}   
 \\\leq  {\frac{4\cdot (2p+1)^{2/3}s^{2\rho}}{n^{2/3}}}+  {\frac{2}{nc}\Delta_1(n^{-\frac{1}{3}})\left(8s+1\right)\cdot(2p+1)}
+ {\frac{4\cdot s^{\rho}\cdot(\Gamma+\frac{2}{n^{1/3}})}{K\cdot n^{1/3}}}\cdot (2p+1)^{1/3}. \label{to use}
\end{multline}
Since \eqref{sample initial requirement 2 3} and $\Delta_1(n^{-\frac{1}{3}})=\ln(np)+\tilde \Delta$, we have
$
\frac{4(2p+1)^{2/3}s^{2\rho}}{n^{2/3}}+\frac{4(\Gamma+\frac{2}{n^{1/3}})\cdot (2p+1)^{1/3}s^{\rho}}{K n^{1/3}}\leq O(1)$ and $ {\frac{2}{nc}\Delta_1(n^{-\frac{1}{3}})\left[ 8s+1\right]\cdot (2p+1)}\leq O(1)$.
Therefore,  it holds that
$
{\frac{2 \tilde p_u}{cn}\Delta_1(n^{-\frac{1}{3}})(2p+1)} \leq O(1)\cdot\sqrt{\frac{(2p+1)^{2/3}s^{2\rho}}{n^{2/3}}+\frac{(\Gamma+\frac{2}{n^{1/3}})\cdot (2p+1)^{1/3}\cdot s^{\rho}}{K n^{1/3}}}
+O(1)\cdot\sqrt{{\frac{\Delta_1(n^{-\frac{1}{3}})}{nc}\cdot\left( 8s+1\right)\cdot (2p+1)}}$.
Combining the above with \eqref{first bound here} and \eqref{to use},   the inequality in \eqref{first results obtained}   can be simplified into
$
\mathbb F(\mathbf X^{RSAA})-\mathbb  F({\mathbf X^*}) \nonumber
\leq  O(1)s^{1-\rho}\cdot \frac{K \cdot \Delta_1(n^{-\frac{1}{3}})\cdot p^{2/3}}{c\cdot n^{2/3}}+ O(1)\cdot K\cdot\sqrt{\frac{p^{2/3}s^{2\rho}}{n^{2/3}}+\frac{(\Gamma+\frac{2}{n^{1/3}})\cdot p^{1/3}s^{\rho}}{K n^{1/3}}} +O(1)\cdot K \sqrt{\frac{s p}{nc}\Delta_1(n^{-\frac{1}{3}})}
+\frac{2}{n^{1/3}}+\Gamma$.
Together with  $\Delta_1(n^{-\frac{1}{3}})\geq \ln 2$, $K \geq 1$, and $0<c\leq 0.5$, the above becomes
\begin{multline}\mathbb F(\mathbf X^{RSAA})-\mathbb  F({\mathbf X^*}) 
\leq O(1)\cdot\left(\frac{s^{1-\rho}\cdot \Delta_1(n^{-1/3})  \cdot p^{2/3}}{n^{2/3}}+\frac{p^{1/3}\cdot s^{\rho}}{n^{1/3}}+\sqrt{\frac{s\cdot p\cdot \Delta_1(n^{-1/3})}{n}}\right)\cdot K 
\\+O(1)\cdot\sqrt{\frac{K \cdot s^{\rho}\cdot p^{1/3}\cdot \Gamma}{n^{1/3}}}+\Gamma,\label{part 1 result here in proof}
\end{multline}
 which then shows  Theorem \ref{first theorem}  since $\Delta_1(n^{-\frac{1}{3}}):=\ln\left( 18n^{1/3} (K_C+\mathcal  C_\mu) \cdot p\cdot R\right)$. \hfill \ensuremath{\blacksquare}

 \subsection{Proof of Corollary \ref{first corollary}}\label{proof of the desired result corollary 1}
 
 Lemma \ref{initial gap} implies that $\mathcal F_{n,\lambda}(\mathbf X^{RSAA},\,{\mathbf Z}_1^n)\leq \mathcal F_{n,\lambda}({\mathbf X}^*,\,{\mathbf Z}_1^n)+\lambda\Vert \mathbf X^*\Vert_*$ almost surely.   Below we invoke the results from Theorem \ref{first theorem} with $\Gamma=\lambda\Vert \mathbf X^*\Vert_*$ and  assumption that $\rho=0$ and $\lambda=\lambda(0)$. Note that it is assumed that \begin{align}
n > C_2\cdot p\cdot \mathcal U_L\cdot [\ln(np)+\tilde \Delta]\cdot s^{3/2}R^{3/2}> O(1)\cdot p\cdot a^{-1}\cdot [\ln(np)+\tilde \Delta]\cdot s^{3/2}R^{3/2},\label{how to use determine}
\end{align}
 and  $\frac{\Gamma}{K }\leq \frac{\lambda\Vert \mathbf X^*\Vert_*}{K}\leq \frac{ \Vert \mathbf X^*\Vert_*\cdot \sqrt{\frac{8K \cdot (2p+1)^{2/3}}{c\cdot a\cdot n^{2/3}}[\ln(n^{1/3}p)+\tilde \Delta]}}{K}$ (as well as $K \geq 1$). In view of \eqref{how to use determine}, it then holds under Assumption 1 that $\frac{\Gamma}{K} \leq R s\cdot \sqrt{\frac{8(2p+1)^{2/3}}{cK\cdot a\cdot n^{2/3}}[\ln(n^{1/3}p)+\tilde \Delta]}\leq O(1)\cdot \sqrt{\frac{Rs}{a^{1/3} }[\ln(n^{1/3}p)+\tilde \Delta]^{1/3}}$. Therefore,  $\left(\frac{\Gamma}{K}\right)^{3}\leq \left( O(1)\cdot  \sqrt{\frac{Rs}{a^{1/3} }[\ln(n^{1/3}p)+\tilde \Delta]^{1/3}}\right)^{3}
\leq O(1)\cdot R^{3/2}s^{3/2}\sqrt{a^{-1}\cdot [\ln(n^{1/3}p)+\tilde \Delta]},$
 for some universal constants $O(1)$. Furthermore, since $a<\mathcal U_L^{-1}\leq1$, it holds that, if $n$ satisfies \eqref{sample initial requirement 2}  for some universal constant $C_2$, then
$
n >  O(1)\cdot p\cdot a^{-1 }\cdot [\ln(n^{1/3}p)+\tilde \Delta]\cdot s^{3/2}R^{3/2}
\geq O(1)\cdot p\cdot R^{3/2}s^{3/2}\sqrt{a^{-1}\cdot [\ln(n^{1/3}p)+\tilde \Delta]}+O(1)\cdot p+ C_1\cdot s\cdot p\cdot \left(\ln(n^{1/3}p)+\tilde \Delta\right)\
\geq C_1\cdot \left[\left(\frac{\Gamma}{K}\right)^{3}p+p+s\cdot p\cdot \left(\ln(n^{1/3}p)+\tilde \Delta\right)\right].$
Therefore,  Theorem \ref{first theorem} is met and thus \eqref{bound global 3} in Theorem \ref{first theorem} implies that 
\begin{multline}
\mathbb F(\mathbf X^{RSAA})-\mathbb  F({\mathbf X^*}) 
\leq O(1)\cdot K\cdot\left(\frac{sp^{2/3}\Delta_1(n^{-1/3})}{n^{2/3}}+\sqrt{\frac{sp\Delta_1(n^{-\frac{1}{3}})}{n}}+\frac{p^{1/3}}{n^{1/3}}\right)
\\+O(1)\cdot\sqrt{\frac{Kp^{1/3}(\lambda\Vert \mathbf X^*\Vert_*)}{n^{1/3}}}+\lambda\Vert \mathbf X^*\Vert_*,\nonumber
\end{multline}
  with probability at least 
$
1-2  (2p+1) \exp(-\tilde cn)-6\exp\left(-2cn^{1/3}\cdot (2p+1)^{2/3} \right)
$. Note that $a<1$, $K \geq 1$, $p\geq 1$, $\left[\ln(n^{1/3}p)+\tilde \Delta\right]\geq 1$ and $\sqrt{\frac{sp\Delta_1(n^{-\frac{1}{3}})}{n}}\leq \frac{s(2p+1)^{1/3}\cdot\sqrt{\Delta_1(n^{-\frac{1}{3}})}}{n^{1/3}}$ (due to \eqref{sample initial requirement 2} again). Hence,
$\mathbb   F(\mathbf X^{RSAA})-\mathbb  F({\mathbf X^*}) 
\leq O(1)\cdot K\cdot\left[\frac{sp^{2/3}\cdot \left(\ln(np)+\tilde \Delta\right)}{n^{2/3}}+\frac{p^{1/3}}{n^{1/3}}\right]
+O(1)\cdot \frac{sRK\cdot(2p+1)^{1/3}}{\min\left\{a^{1/2}n^{1/3},\,a^{1/4}n^{1/3}\right\}} \left[\ln(n^{1/3}p)+\tilde \Delta\right]^{1/2}$, which shows Part (ii) by further noticing that $a=\frac{1}{2\mathcal U_L}$ and $\mathcal U_L\geq 1$.
 \hfill \ensuremath{\blacksquare}
 \subsection{Proof of Corollary \ref{second corollary}}\label{proof of the desired result corollary 2}
The proof follows almost the same argument as in Section \ref{proof of the desired result corollary 1} for proving Corollary \ref{first corollary}, except that the choice of user-specific parameters are different. Again, Lemma \ref{initial gap} implies that $\mathcal F_{n,\lambda}(\mathbf X^{RSAA},\,{\mathbf Z}_1^n)\leq \mathcal F_{n,\lambda}({\mathbf X}^*,\,{\mathbf Z}_1^n)+\lambda\Vert \mathbf X^*\Vert_*$ almost surely.   As the same in Part (ii), below we invoke the results from Theorem \ref{first theorem}   with $\Gamma=\lambda\Vert \mathbf X^*\Vert_*$ and  assumption that $\rho=2/3$ and $\lambda=\lambda(\frac{2}{3})$. Note that it is assumed that \begin{align}
n > C_3\cdot p\cdot \mathcal U_L\cdot [\ln(np)+\tilde \Delta]\cdot s^{2}R^{3/2}> O(1)\cdot p\cdot a^{-1}\cdot [\ln(np)+\tilde \Delta]\cdot s^{2}R^{3/2},\label{test to use here how}
\end{align}
 and  $\frac{\Gamma}{K }\leq \frac{\lambda\Vert \mathbf X^*\Vert_*}{K}\leq \frac{ \Vert \mathbf X^*\Vert_*\cdot \sqrt{\frac{8K \cdot (2p+1)^{2/3}\cdot s^{-2/3}}{c\cdot a\cdot n^{2/3}}[\ln(n^{1/3}p)+\tilde \Delta]}}{K}$ (as well as $K \geq 1$). In view of \eqref{test to use here how}, it then holds under Assumption 1 that $\frac{\Gamma}{K} \leq R s\cdot \sqrt{\frac{8(2p+1)^{2/3} s^{-2/3}}{cK\cdot a\cdot n^{2/3}}[\ln(n^{1/3}p)+\tilde \Delta]}\leq O(1)\cdot \sqrt{\frac{R}{a^{1/3} }[\ln(n^{1/3}p)+\tilde \Delta]^{1/3}}$. Therefore,  $\left(\frac{\Gamma}{K}\right)^{3}\leq \left( O(1)\cdot  \sqrt{\frac{R}{a^{1/3} }[\ln(n^{1/3}p)+\tilde \Delta]^{1/3}}\right)^{3}
\leq O(1)\cdot R^{3/2} \sqrt{a^{-1}\cdot [\ln(n^{1/3}p)+\tilde \Delta]},$
 for some universal constants $O(1)$. Furthermore, since $a<\mathcal U_L^{-1}\leq1$, it holds that, if $n$ satisfies \eqref{sample initial requirement 2 new 234}, then
$
n >  O(1)\cdot p\cdot a^{-1 }\cdot [\ln(n^{1/3}p)+\tilde \Delta]\cdot s^{2}R^{3/2}
\geq O(1)\cdot p\cdot R^{3/2}s^{2}\sqrt{a^{-1}\cdot [\ln(n^{1/3}p)+\tilde \Delta]}+O(1)\cdot s^2\cdot p+ C_1s\cdot p \cdot\left(\ln(n^{1/3}p)+\tilde \Delta\right)\
\geq C_1\cdot \left[ s^2\left(\frac{\Gamma}{K}\right)^{3}p+s^2\cdot p+s\cdot p\cdot \left(\ln(n^{1/3}p)+\tilde \Delta\right)\right].$
Therefore, \eqref{sample initial requirement 2 3} in Theorem \ref{first theorem} is met and thus \eqref{bound global 3} in Theorem \ref{first theorem} implies that 
\begin{multline}
\mathbb F(\mathbf X^{RSAA})-\mathbb  F({\mathbf X^*}) 
\leq O(1)\cdot K\cdot\left(\frac{s^{1/3}p^{2/3}\Delta_1(n^{-1/3})}{n^{2/3}}+\sqrt{\frac{sp\Delta_1(n^{-\frac{1}{3}})}{n}}+\frac{p^{1/3}\cdot s^{2/3}}{n^{1/3}}\right)
\\+O(1)\cdot\sqrt{\frac{Kp^{1/3}\cdot s^{2/3}\cdot (\lambda\Vert \mathbf X^*\Vert_*)}{n^{1/3}}}+\lambda\Vert \mathbf X^*\Vert_*,\nonumber
\end{multline}
  with probability at least 
$
1-2  (2p+1) \exp(-\tilde cn)-6\exp\left(-2cn^{1/3}\cdot (2p+1)^{2/3} \right)
$. Note that $a<1$, $K \geq 1$, $p\geq s\geq 1$, $\left[\ln(n^{1/3}p)+\tilde \Delta\right]\geq 1$ and $\sqrt{\frac{sp\Delta_1(n^{-\frac{1}{3}})}{n}}\leq \frac{(2p+1)^{1/3}\cdot\sqrt{\Delta_1(n^{-\frac{1}{3}})}}{n^{1/3}}$ (in view of \eqref{sample initial requirement 2 new 234} again). Hence,
$\mathbb   F(\mathbf X^{RSAA})-\mathbb  F({\mathbf X^*}) 
\leq O(1)\cdot K\cdot\left[\frac{s^{1/3}p^{2/3}\cdot \left(\ln(np)+\tilde \Delta\right)}{n^{2/3}}+\frac{s^{2/3}\cdot p^{1/3}}{n^{1/3}}\right]
+O(1)\cdot \frac{s^{2/3}RK\cdot(2p+1)^{1/3}}{\min\left\{a^{1/2}n^{1/3},\,a^{1/4}n^{1/3}\right\}} \left[\ln(n^{1/3}p)+\tilde \Delta\right]^{1/2}$, which shows Part (iii) by further noticing that $a=\frac{1}{2\mathcal U_L}$.
 \hfill \ensuremath{\blacksquare}

  \subsection{Auxiliary results}
  
  \begin{proposition}\label{Proposition 1}

Suppose that  $a<{U_L}^{-1}$. Assume that the {\it S$^3$ONC}$(\mathbf Z_1^n)$  is satisfied almost surely at $\mathbf X^{RSAA}\in\mathcal S_p  $. Then,
$$\mathbb P[\{\text{$\vert\sigma_j(\mathbf X^{RSAA})\vert\notin(0,\,a\lambda)$ for all $j$}\}]=1.$$
\end{proposition}
\begin{proof}
 Since $\mathbf X^{RSAA}$  satisfies the {\it S$^3$ONC}$(\mathbf Z_1^n)$  almost surely, Eq.\,\eqref{second condition 1} implies that for any $j\in\{1,...,p\}$, if $\sigma_j(\mathbf X^{RSAA})\in(0,\,a\lambda)$, then
\begin{align}
0\leq&\,U_L+\left[\frac{\partial^2 P_\lambda(\vert\sigma_j(\mathbf X)\vert)}{[\partial\sigma_j(\mathbf X)]^2}\right]_{\mathbf X=\mathbf X^{RSAA}} 
=  U_L-\frac{1}{a}.\label{to contradict 0 alambda}
\end{align}
Further observe that  $\frac{\partial^2 P_\lambda(t)}{\partial t^2}=-a^{-1}$ for $t\in(0,\,a\lambda)$. Therefore, \eqref{to contradict 0 alambda} contradicts with the assumption that $U_L<\frac{1}{a}$.  This contradiction implies that 
\begin{align}
&\,\mathbb P[\{\mathbf X^{RSAA}\text{ satisfies the S$^3$ONC$(\mathbf Z_1^n)$}\}\cap\{\vert\sigma_j(\mathbf X^{RSAA})\vert\in(0,\,a\lambda)\}]=0\nonumber
\\\Longrightarrow &\,0\geq 1-\mathbb P[\{\mathbf X^{RSAA}\text{ does not satisfy the S$^3$ONC$(\mathbf Z_1^n)$}\}]-\mathbb P[\{\vert\sigma_j(\mathbf X^{RSAA})\vert\notin(0,\,a\lambda)\}].\nonumber
\end{align}
 Since  $\mathbb P[\{\mathbf X^{RSAA}\text{ satisfies the S$^3$ONC$(\mathbf Z_1^n)$}\}]=1$, it holds that $\mathbb P[\{\vert\sigma_j(\mathbf X^{RSAA})\vert\notin(0,\,a\lambda)\}]=1$ for all $j=1,...,n$, which immediately leads to the desired result.
 \end{proof}


\begin{proposition}\label{test proposition important} Suppose that Assumptions \ref{sub exponential condition} and \ref{Lipschitz condition} hold.
 Let $\epsilon\in(0,\, 1]$, $\tilde p:\,\tilde p> s$,  $\Delta_1(\epsilon):=\ln\left(\frac{18\cdot{(K_C+\mathcal  C_\mu)}\cdot p\cdot R}{\epsilon}\right)$, and  $\mathcal B_{\tilde p,R}:=\left\{\mathbf X\in\mathcal S_p:\, \sigma_{\max}(\mathbf X)\leq R,\, {\bf rk}(\mathbf X)\leq \tilde p \right\}.$  Then, for the same $c\in(0,\,0.5]$ as in \eqref{sub-exponential sequence sum} and for some  $\tilde c>0$,
$$
\max_{\mathbf X\in \mathcal B_{\tilde p,R}}\left\vert\frac{1}{n}\sum_{i=1}^n f(\mathbf X,Z_i)-\mathbb   F(\mathbf X)\right\vert
\leq  {\frac{K }{\sqrt{n}}}\sqrt{\frac{2 \tilde p(2p+1)}{c}\Delta_1(\epsilon)} +\frac{K}{n}\cdot\frac{2 \tilde p(2p+1)}{c}\Delta_1(\epsilon)+\epsilon$$ 
with probability at least $1-2\exp\left(- \tilde p(2p+1)\Delta_1(\epsilon)\right)-2\exp(-\tilde cn)$.
\end{proposition}
\begin{proof}
We will follow the ``$\epsilon$-net''  argument similar to \citet{Shapiro:book} to construct a net of discretization grids $\mathcal G(\epsilon):=\{\tilde{\mathbf X}^k\}\subseteq \mathcal B_{\tilde p,R} $ such that for any $\mathbf X\in \mathcal B_{\tilde p,R} $, there is  $\mathbf X^k\in \mathcal G(\epsilon)$ that satisfies $\Vert\mathbf X^k-\mathbf X\Vert\leq \frac{\epsilon}{2K_C+2\mathcal  C_\mu}$  for any fixed $\epsilon\in(0,\,1]$. 

Invoking Lemma \ref{low rank epsilon net}, for  an arbitrary  $\mathbf X\in\mathcal B_{\tilde p,R}$, to ensure  that there always exists $\tilde{\mathbf X}^k \in \mathcal G(\epsilon)$ that ensures $\left\Vert\mathbf X-\tilde{\mathbf X}^k\right\Vert\leq \frac{\epsilon}{ (2K_C+2\mathcal  C_\mu)}$,
it is sufficient to have the number of grids to be no more than
$\left(\frac{18R\sqrt{\tilde p}\cdot (K_C+\mathcal C_\mu)}{\epsilon}\right)^{(2p+1)\tilde p}$. 
Now, we may observe
\begin{align}
&\mathbb P\left[\max_{\mathbf X^k\in \mathcal G(\epsilon)}\left\vert\frac{1}{n}\sum_{i=1}^n f(\mathbf X^k,Z_i)-\mathbb E\left[\frac{1}{n}\sum_{i=1}^n f(\mathbf X^k,Z_i)\right]\right\vert
\leq K \sqrt{\frac{ t}{n}}
+\frac{K t}{n} \right]  \nonumber
\\=&\mathbb P\left[\bigcap_{\mathbf X^k\in \mathcal G(\epsilon)}\left\{\left\vert\frac{1}{n}\sum_{i=1}^n f(\mathbf X^k,Z_i)-\mathbb E\left[\frac{1}{n}\sum_{i=1}^n f(\mathbf X^k,Z_i)\right]\right\vert
\leq K \sqrt{\frac{ t}{n}}
+\frac{K t}{n} \right\}\right]  \nonumber
\\\geq&1-\sum_{\mathbf X^k\in \mathcal G(\epsilon)}\mathbb P\left[\left\vert\frac{1}{n}\sum_{i=1}^n f(\mathbf X^k,Z_i)-\mathbb E\left[\frac{1}{n}\sum_{i=1}^n f(\mathbf X^k,Z_i)\right]\right\vert
> K \sqrt{\frac{ t}{n}}
+\frac{K t}{n}\right].  \label{to use here}
\end{align}
Further invoking   Eq.\,\eqref{sub-exponential sequence sum}, for  the same $c$ as in \eqref{sub-exponential sequence sum}, it holds that
\begin{align}
&\,\mathbb P\left[\max_{\mathbf X^k\in \mathcal G(\epsilon)}\left\vert\frac{1}{n}\sum_{i=1}^n f(\mathbf X^k,Z_i)-\mathbb E\left[\frac{1}{n}\sum_{i=1}^n f(\mathbf X^k,Z_i)\right]\right\vert
\leq K \sqrt{\frac{ t}{n}}
+\frac{K t}{n} \right] \nonumber
\\\geq &\,1-\vert \mathcal G(\epsilon)\vert\cdot2\exp(-ct)  \geq \,1-2\left(\frac{18R\sqrt{\tilde p}\cdot (K_C+\mathcal C_\mu)}{\epsilon}\right)^{(2p+1)\tilde p} \cdot\exp(-ct).\nonumber
\end{align}
Combined with Lemma  \ref{new lemma result} and Lemma \ref{expected cost result Matrix},
\begin{multline}\max_{\substack{{\mathbf X}\in\mathcal B_{\tilde p, R},\,\mathbf X^k\in \mathcal G(\epsilon)}}\left\{\left\vert \frac{1}{n}\sum_{i=1}^n f(\mathbf X,Z_i) -\frac{1}{n}\sum_{i=1}^n f(\mathbf X^k,Z_i) \right\vert \right.
\\\left. +\left\vert \mathbb E\left[\frac{1}{n}\sum_{i=1}^n f(\mathbf X,Z_i)\right] -\mathbb E\left[\frac{1}{n}\sum_{i=1}^n f(\mathbf X^k,Z_i) \right]\right\vert\right\}
\\\leq 2(K_C+\mathcal C_{\mu})\cdot \frac{\epsilon}{2K_C+2\mathcal  C_\mu}=\epsilon,
\end{multline}
 with probability at least $1-2\exp(-\tilde c\cdot n)$ for some problem independent $\tilde c>0$ and any fixed $\tau>0$. Observe that for any $\mathbf X\in \mathcal B_{\tilde p,R}$ and $\mathbf X^k\in \mathcal G(\epsilon)$, it holds that 
$\left\vert\mathcal F_n(\mathbf X,\mathbf Z_1^n)-\mathbb E\left[\mathcal F_n(\mathbf X,\mathbf Z_1^n)\right]\right\vert   
\leq \left\vert\mathcal F_n(\mathbf X^k,\mathbf Z_1^n)-\mathbb E\left[\mathcal F_n(\mathbf X^k,\mathbf Z_1^n)\right]\right\vert  
+ \left\vert \mathcal F_n(\mathbf X,\mathbf Z_1^n) -\mathcal F_n(\mathbf X^k,\mathbf Z_1^n) \right\vert  +\left\vert \mathbb E\left[\mathcal F_n(\mathbf X,\mathbf Z_1^n)\right] -\mathbb E\left[\mathcal F_n(\mathbf X^k,\mathbf Z_1^n) \right]\right\vert.$
Therefore,  with probability at least $1-2\exp(-\tilde c\cdot n)$ for some positive constant $\tilde c>0$,
\begin{multline}
\max_{\substack{{\mathbf X}\in\mathcal B_{\tilde p, R},\,\mathbf X^k\in \mathcal G(\epsilon)}} \left\{\left\vert\mathcal F_n(\mathbf X,\mathbf Z_1^n)-\mathbb E\left[\mathcal F_n(\mathbf X,\mathbf Z_1^n)\right]\right\vert-\left\vert\mathcal F_n(\mathbf X^k,\mathbf Z_1^n)-\mathbb E\left[\mathcal F_n(\mathbf X^k,\mathbf Z_1^n)\right]\right\vert\right\} \leq  \epsilon.\label{to minimize right}
\end{multline}
Further invoking \eqref{to use here}, we now obtain that
 \begin{align}
\max_{\substack{{\mathbf X}\in\mathcal B_{\tilde p, R},\,\mathbf X^k\in \mathcal G(\epsilon)}}\left\vert\mathcal F_n(\mathbf X,\mathbf Z_1^n)-\mathbb F(\mathbf X)\right\vert   
\leq \epsilon+ K \sqrt{\frac{ t}{n}}
+\frac{K t}{n}, \nonumber
\end{align}
with probability at least $1-2\left(\frac{18R\sqrt{\tilde p}\cdot (K_C+\mathcal C_\mu)}{\epsilon}\right)^{(2p+1)\tilde p} \cdot\exp(-ct)-2\exp(-\tilde c\cdot n)$.  
Finally, we may let $t:=\frac{2\tilde p}{c}\cdot (2p+1)\cdot \Delta_1(\epsilon)$, where $\Delta_1(\epsilon):=\ln\left(\frac{18\cdot  (K_C+\mathcal  C_\mu)\cdot p\cdot R}{\epsilon}\right)$, and obtain the desired result.
   \end{proof} 
\begin{proposition}\label{bound dimensions}
 Suppose that Assumptions \ref{approximate low-rankness} through \ref{sub exponential condition}  hold, the solution $\mathbf X^{RSAA}\in\mathcal S_p:\,\sigma_{\max}(\mathbf X^{RSAA})\leq R$ satisfies {\it S$^3$ONC}$(\mathbf Z_1^n)$ almost surely, 
 \begin{align}
& \mathcal F_{n,\lambda}(\mathbf X^{RSAA},\mathbf Z_1^n)\leq \mathcal F_{n,\lambda}({\mathbf X}^*,\mathbf Z_1^n)+\Gamma,~~w.p.1.\label{epsilon delta}
\end{align}
where  $\Gamma\geq 0$, $\epsilon\in(0,\, 1]$,  $
\Delta_1(\epsilon):=\ln\left(\frac{18\cdot{(K_C+\mathcal  C_\mu)}\cdot p\cdot R}{\epsilon}\right)$.
For a positive integer $\tilde p_u:\,\tilde p_u>s$, if 
\begin{align}
(\hat p-s)\cdot P_\lambda(a\lambda) > \frac{4K}{cn}\Delta_1(\epsilon)\cdot  \hat p\cdot (2p+1)
+{\frac{2K}{\sqrt{n}}}\sqrt{\frac{2\hat p\cdot (2p+1)}{c}\Delta_1(\epsilon)}+\Gamma+2\epsilon,\label{condition on value lambda 1}
\end{align}
for all $\hat p:\,  \tilde p_u\leq \hat p\leq p$,  then 
$
\mathbb P[{\bf rk}(\mathbf X^{RSAA})\leq \tilde p_u-1]\geq \,1-2 p \exp(-\tilde cn)-4\exp\left(- \tilde p_u(2p+1)\Delta_1(\epsilon)\right)\nonumber
$ for the same $c$ in \eqref{sub-exponential sequence sum} and some $\tilde c>0$.
\end{proposition}
\begin{proof} This proof generalizes Proposition EC.3 from \cite{Liu2019} bounding the sparsity of an S$^3$ONC solution to bounding the rank of an S$^3$ONC solution. Though the argument is similar, details are quite different and thus the result is different.
Define $\mathcal B_{R}:=\{\mathbf X\in\mathcal S_p:\,\sigma_{\max}(\mathbf X)\leq R\}$. Define a few events:
\begin{align}\mathcal E_{1}:=&\,\left\{(\tilde{\mathbf X},\,\tilde{\mathbf Z}_1^n)\in\mathcal B_{R}\times \mathcal W^n:\,\mathcal F_{n,\lambda}(\tilde{\mathbf X},\,\tilde{\mathbf Z}_1^n)\leq \mathcal F_{n,\lambda}({\mathbf X}^*,\,\tilde{\mathbf Z}_1^n)+\Gamma\right\},\nonumber
\\\mathcal E_2:=&\,\{\tilde{\mathbf X}\in\mathcal B_{R}:\, \text{$\vert\sigma_j(\tilde{\mathbf X})\vert\notin(0,\,a\lambda)$ for all $j$}\},\nonumber
\\
\mathcal E_{3,\hat p}:=&\,\left\{ \tilde{\mathbf X}\in\mathcal B_{R}:\,{\bf rk}(\tilde{\mathbf X})=\hat p\right\}, \nonumber
\nonumber
\end{align}
where $c$ in $\mathcal E_{5,\hat p}$ is a universal constant defined to be the same as in \eqref{sub-exponential sequence sum}, $\hat p:\,  \tilde p_u\leq\hat p\leq p$ and (thus $\hat p>s$ by the assumption that $\tilde p_u>s$).
For any $(\tilde{\mathbf X},\tilde{\mathbf Z}_1^n)\in\{(\tilde{\mathbf X},\tilde{\mathbf Z}_1^n)\in \mathcal E_{1}\}\cap\{\tilde{\mathbf X}\in \mathcal E_2\cap \mathcal E_{3,p}\}$, where $\tilde{\mathbf Z}_1^n=(\tilde  Z_1,...,\tilde Z_n)$, since   $\tilde{\mathbf X}\in \mathcal E_{3,p}\cap \mathcal E_2$, which means that  $\tilde{\mathbf X}$ has     $\hat p$-many   non-zero singular values and each must not be within the interval $(0,\,a\lambda)$, it holds that
\begin{align}
\mathcal F_n(\tilde{\mathbf X},\tilde{\mathbf Z}_1^n) +\hat pP_\lambda(a\lambda)\leq \frac{1}{n}  \mathcal F_n({\mathbf X^*},\tilde{\mathbf Z}_1^n) +s P_\lambda(a\lambda)+\Gamma,\label{first to combine here}
\end{align}
Notice   that $\mathbf X^*\in\mathcal B_R:\,{\bf rk}(\mathbf X^*)=s<\hat p$ by Assumption \ref{approximate low-rankness}. We may obtain that, for all $\tilde{\mathbf X} \in   \mathcal E_{3,p}$, 
\begin{align}
&\,\frac{1}{n}\sum_{i=1}^n f({\mathbf X}^*,\tilde Z_i) -\frac{1}{n}\sum_{i=1}^nf(\tilde{\mathbf X},\tilde Z_i) \nonumber
\\=&\, \left[\frac{1}{n}\sum_{i=1}^n f({\mathbf X}^*,\tilde Z_i) -\mathbb  F({\mathbf X^*})  \right]+ \left[\mathbb  F (\tilde{\mathbf X})  
-\frac{1}{n}\sum_{i=1}^nf(\tilde{\mathbf X},\tilde Z_i) \right]
+ \left[\mathbb  F ({\mathbf X^*} ) -\mathbb  F(\tilde{\mathbf X}) \right]\nonumber
\\\leq& \, 2\max_{\mathbf X\in   \mathcal E_{3,p}} \left\vert\frac{1}{n}\sum_{i=1}^n f(\mathbf X,\tilde Z_i)-\mathbb   F(\mathbf X )  \right\vert +\mathbb  F({\mathbf X^*}) -\mathbb  F(\tilde{\mathbf X}) \nonumber
\\\leq &\, 2\max_{\mathbf X\in   \mathcal E_{3,p}} \,\,\left\vert\frac{1}{n}\sum_{i=1}^n f(\mathbf X,\tilde Z_i)-\mathbb  F({\mathbf X})  \right\vert, \label{resolution inequality}
\end{align}
where the last inequality is due to $\mathbb  F({\mathbf X^*}) \leq \mathbb  F({\mathbf X})$ for all $\mathbf X\in\mathcal S_p$ by the definition of $\mathbf X^*$.
Define that
\begin{align}
\mathcal E_{4}:=&\,\left\{(\tilde{\mathbf X},\,\tilde{\mathbf Z}_1^n)\in\mathcal B_{R}\times\mathcal W^n:\,\text{$\tilde{\mathbf X}$ satisfies  S$^3$ONC$(\tilde{\mathbf Z}_1^n)$}\right\}\nonumber
\\
\mathcal E_{5,\hat p}:=&\,\left\{\tilde{\mathbf Z}_1^n\in\mathcal W^n:\,\max_{\mathbf X\in\mathcal B_{R}:\,{\bf rk}(\mathbf X)\leq \hat p}\left\vert\mathcal F_n(\mathbf X,\tilde{\mathbf Z}_1^n)-\mathbb  F (\mathbf X)\right\vert\leq  {\frac{K }{\sqrt{n}}}\sqrt{\frac{2 \hat p(2p+1)}{c}\Delta_1(\epsilon)} \right.\nonumber
\\&
\left.\qquad\qquad\qquad\qquad\qquad\qquad\qquad\qquad\qquad\qquad\qquad+\frac{K}{n}\cdot\frac{2 \hat p(2p+1)}{c}\Delta_1(\epsilon)+\epsilon\right\},\nonumber
\end{align}

Now let us examine the following set:
\begin{align}\Lambda=&\{(\tilde{\mathbf X},\tilde{\mathbf Z}_1^n):\,(\tilde{\mathbf X},\tilde{\mathbf Z}_1^n)\in\mathcal E_{1}\cap \mathcal E_{4}\}\cap\{(\tilde{\mathbf X},\tilde{\mathbf Z}_1^n):\,\tilde{\mathbf X} \in\mathcal E_{3,p}\cap \mathcal E_2\}\cap\{(\tilde{\mathbf X},\tilde{\mathbf Z}_1^n):\,\tilde{\mathbf Z}_{1}^n\in \mathcal E_{5,\hat p}\}.\nonumber
\end{align} 
Combined with \eqref{first to combine here} and \eqref{resolution inequality},  $\Lambda\neq \emptyset\Longrightarrow
(\hat p-s)\cdot P_\lambda(a\lambda)\leq {\frac{2K }{\sqrt{n}}}\sqrt{\frac{2 \hat p(2p+1)}{c}\Delta_1(\epsilon)} +\frac{2K}{n}\cdot\frac{2 \hat p(2p+1)}{c}\Delta_1(\epsilon)+2\epsilon+\Gamma$, which contradicts with \eqref{condition on value lambda 1} for all $\hat p:\, \tilde p_u\leq \hat p\leq p$.  Now  we recall the definition of  $\mathbf X^{RSAA}\in \mathcal B_{R}$, which is  a solution that satisfies the {\it S$^3$ONC}$(\mathbf Z_1^n)$, w.p.1., and $\mathcal F_{n,\lambda}(\mathbf X^{RSAA},\,\tilde{\mathbf Z}_1^n)\leq \mathcal F_{n,\lambda}({\mathbf X}^*,\,\tilde{\mathbf Z}_1^n)+\Gamma$, w.p.1. Invoking Proposition \ref{Proposition 1}, we have  $\mathbb P\left[(\mathbf X^{RSAA},{\mathbf Z}_1^n)\in\mathcal E_{1}\cap \mathcal E_{4},\, \mathbf X^{RSAA} \in \mathcal E_2\right]=1$. Hence,
 \begin{align}
 0=&\,\mathbb P\left[\Lambda\right]\nonumber
 \\\geq &\,1-\mathbb P\left[ \mathbf X^{RSAA} \notin \mathcal E_{3,p}  \right]-\mathbb P\left[\mathbf Z_1^n \notin \mathcal E_{5,\hat p}\right]-\left\{1-\mathbb P\left[(\mathbf X^{RSAA},{\mathbf Z}_1^n)\in\mathcal E_{1}\cap \mathcal E_{4},\, \mathbf X^{RSAA} \in \mathcal E_2 \right]\right\},
\nonumber
 \end{align} 
 for all $\hat p:\, \tilde p_u\leq \hat p\leq p$.  The above then implies that
$\mathbb P\left[\mathbf Z_1^n \notin \mathcal E_{5,\hat p} \right]\geq \mathbb P\left[ \mathbf X^{RSAA} \in \mathcal E_{3,p} \right]$ for all $\hat p:\, \tilde p_u\leq \hat p\leq p$.
Therefore,   
$\mathbb P[{\bf rk}(\mathbf X^{RSAA})= \hat p]\leq 1-\mathbb P\left[ \mathbf Z_1^n \in \mathcal E_{5,\hat p}\right]$
for all $\hat p:\, \tilde p_u\leq \hat p\leq p$. Together with Proposition \ref{test proposition important}, we have that
\begin{align}
&\mathbb P[{\bf rk}(\mathbf X^{RSAA})\leq \tilde p_u-1]=\mathbb P[{\bf rk}(\mathbf X^{RSAA})\notin\{\tilde p_u,\, \tilde p_u+1,..., p\}] \nonumber
\\= &1-\mathbb P\left[\bigcup_{\hat p=\tilde p_u}^{p}\{{{\bf rk}(\mathbf X^{RSAA})}=\hat p\}\right]\geq 1- \sum_{\hat p=\tilde p_u}^{p} \mathbb P[{\bf rk}(\mathbf X^{RSAA})=\hat p]
\geq  1-\sum_{\hat p=\tilde p_u}^{p}\left(1-\mathbb P\left[\mathbf Z_1^n \in \mathcal E_{5,\hat p}\right]\right) \nonumber
\\\geq & 1-2({p}-\tilde p_u+1)\exp(-\tilde cn)-\sum_{\hat p=\tilde p_u}^{p} 2\exp\left(- \hat p(2p+1)\cdot\Delta_1(\epsilon)\right).\label{geometric sequence}
\end{align}
where $\tilde c>0$ is some universal constant. Observing that $\Delta_1(\epsilon)=\ln\left(\frac{18\cdot  (K_C+\mathcal  C_\mu)\cdot p\cdot R}{\epsilon}\right)>1$  by observing that the above \eqref{geometric sequence} involves a geometric sequence, we have
\begin{align}
\mathbb P[{\bf rk}(\mathbf X^{RSAA})\leq \tilde p_u-1]\geq & 1-\frac{2\exp\left(- \tilde p_u (2p+1)\Delta_1(\epsilon)\right)}{1- \exp\left(-(2p+1)\Delta_1(\epsilon)\right)}-2 {p} \exp(-\tilde cn). \label{starting point}
\end{align}
Further noting that $\frac{2\exp\left(- \tilde p_u (2p+1)\Delta_1(\epsilon)\right)}{1- \exp\left(-(2p+1)\Delta_1(\epsilon)\right)}\leq 4\exp\left(- \tilde p_u(2p+1)\Delta_1(\epsilon)\right)$, we then have the desired result.  
 \end{proof}


\begin{proposition}\label{core proposition 1}
Let   $$\Delta_1(\epsilon):=\ln\left(\frac{18\cdot  (K_C+\mathcal  C_\mu)\cdot p\cdot R}{\epsilon}\right).$$
Assume that (i) the solution $\mathbf X^{RSAA}$ satisfies {\it S$^3$ONC}$(\mathbf Z_1^n)$ almost surely; (ii) $\mathcal F_{n,\lambda}(\mathbf X^{RSAA},\,{\mathbf Z}_1^n)\leq \mathcal F_{n,\lambda}({{\mathbf X}}^*,\,{\mathbf Z}_1^n)+\Gamma$ with probability one;  and (iii)  for some integer $\tilde p_u:\, \tilde p_u>s$, it holds that
\begin{align}
\hat p> s+\frac{4K\cdot  \hat p\cdot (2p+1)}{cn\cdot P_\lambda(a\lambda) }\Delta_1(\epsilon)
+{\frac{2K}{\sqrt{n}\cdot P_\lambda(a\lambda) }}\sqrt{\frac{2\hat p\cdot (2p+1)}{c}\Delta_1(\epsilon)}+\frac{\Gamma+2\epsilon}{P_\lambda(a\lambda)},\label{condition on value lambda}
\end{align}
for all $\tilde p:\,  \tilde p_u\leq \tilde p\leq p$, any $\Gamma\geq 0$,  and any $\epsilon\in(0,\,1]$.
It then holds that
\begin{multline}
\mathbb F(\mathbf X^{RSAA})-\mathbb  F({{\mathbf X}}^*)
\leq \frac{4K\cdot  \hat p\cdot (p+1)}{cn}\Delta_1(\epsilon)
\\+{\frac{2K}{\sqrt{n}}}\sqrt{\frac{2\hat p\cdot (2p+1)}{c}\Delta_1(\epsilon)}+\Gamma+2\epsilon+s  P_\lambda(a\lambda), \label{first results}
\end{multline}
with probability at least
$
P^*:= \,1-2 (p+1) \exp(-\tilde cn)
-6\exp\left(- \tilde p_u(2p+1)\Delta_1(\epsilon)\right)$ for some universal constant $\tilde c>0$.
\end{proposition}

\begin{proof}   
We first observe that $\Delta_1(\epsilon):=\ln\left(\frac{18 \cdot  (K_C+\mathcal  C_\mu)\cdot p\cdot  R}{\epsilon}\right)\geq\ln 36$ because  $p\geq 1$,  ${K}_C,\,C_\mu,\,R\geq1$ and  $0<\epsilon\leq 1$. By assumption, $$\mathcal F_{n,\lambda}(\mathbf X^{RSAA},\,{\mathbf Z}_1^n)\leq \mathcal F_{n,\lambda}({\mathbf X}^*,\,{\mathbf Z}_1^n)+\Gamma,$$  w.p.1., $P_\lambda(t)\geq 0$ for all $t\geq 0$, and   ${\bf rk}(\mathbf X^*) =s$, yields that
$\frac{1}{n}\sum_{i=1}^n f(\mathbf X^{RSAA},Z_i) 
 \leq  \frac{1}{n}  \sum_{i=1}^n f({\mathbf X^*},Z_i) +s P_\lambda(a\lambda)+\Gamma, \,a.s.$
Furthermore, conditioning on the events that (a) ${\bf rk}(\mathbf X^{RSAA})\leq \tilde p_u$,  (b) 
$
  \max_{\mathbf X\in \mathcal B_{\tilde p_u,R}}\left\vert\frac{1}{n}\sum_{i=1}^n f(\mathbf X,Z_i)-\mathbb E\left[\frac{1}{n}\sum_{i=1}^n f(\mathbf X,Z_i)\right]\right\vert  
\leq {\frac{K}{\sqrt{n}}}\sqrt{\frac{\tilde p_u\cdot (2p+1)}{c}\Delta_1(\epsilon)}  
 +\frac{K}{n}\frac{\tilde p_u\cdot (2p+1)}{c}\Delta_1(\epsilon)+\epsilon,$
 we obtain that
$
\mathbb F(\mathbf X^{RSAA})-\mathbb F({\mathbf X^*}) 
\leq s\cdot P_\lambda(a\lambda)+ {\frac{2K }{\sqrt{n}}}\sqrt{\frac{2 \tilde p_u\cdot (2p+1)}{c}\Delta_1(\epsilon)} \nonumber
+\frac{4K}{n}\frac{\tilde p_u\cdot (2p+1)}{c}\Delta_1(\epsilon)+2\epsilon+\Gamma$, a.s.
Further invoking Propositions \ref{test proposition important} and \ref{bound dimensions}, we have that  both events  hold simultaneously with probability at least as in $P^*$, which verifiably implies the claimed results.
\end{proof}

\subsection{Useful Lemmata}
 \begin{lemma}\label{new lemma result}
Under Assumption \ref{Lipschitz condition}, it holds that, for some universal constant $c>0$,  with probability at least $1-2\exp(-c\cdot n)$, it holds  that
\begin{align}
&\max_{\substack{{\mathbf X}_1,\,{\mathbf X}_2\in\mathcal S_p
\\\cap \{{\mathbf X}:\,\sigma_{\max}({\mathbf X})\leq R,\\
\,\Vert\mathbf X_1-\mathbf X_2\Vert\leq \tau\}
}}\{\vert \mathcal F_n({\mathbf X}_1, \mathbf Z_1^n)-\mathcal F_n({\mathbf X}_2, \mathbf Z_1^n)\vert\}\leq \left(2 K_C+\mathcal  C_\mu\right)\cdot \tau.\nonumber
\end{align}
for any given $\tau\geq 0$.
\end{lemma}
 \begin{proof}  This proof follows a closely similar lemma by \cite{Shapiro:book}. Similar proof has also been provided by \cite{Liu2019}, but some subtle differences in the problem context present and thus we redo the the proof herein.
By Assumption \ref{Lipschitz condition},   for some $c>0$, 
\begin{align}
\mathbb P\left(\left\vert \sum_{i=1}^n \frac{1}{n}\left\{\mathcal C(Z_i)-\mathbb E[\mathcal C(Z_i)] \right\}\right\vert>K_C\left(\frac{t}{n}+\sqrt{\frac{t}{n}}\right)\right)\leq 2\exp\left(-c t\right),\qquad\forall t\geq 0.\nonumber
\end{align}
If we let $t:=n$ and observe that $\mathbb E[\mathcal C(Z_i)]\leq \mathcal  C_\mu$, we immediately have that
\begin{align}
\mathbb P\left(\sum_{i=1}^n\frac{\mathcal C(Z_i)}{n}\leq 2 K_C+\mathcal  C_\mu\right)\leq 1-2\exp\left(-c n\right).\label{prob lipschitz bound}
\end{align}
If we invoke Assumption \ref{Lipschitz condition} again  given the event that $\left\{\sum_{i=1}^n\frac{\mathcal C(Z_i)}{n}\leq 2 K_C+\mathcal  C_\mu\right\}$, we have that for any ${\mathbf X}_1,{\mathbf X}_2\in\mathcal S_p$,
\begin{align}
&\max_{\substack{{\mathbf X}_1,\,{\mathbf X}_2\in\mathcal S_p
\\\cap \{{\mathbf X}:\,\sigma_{\max}({\mathbf X})\leq R,\\
\,\Vert\mathbf X_1-\mathbf X_2\Vert\leq \tau\}
}}\left\vert  \frac{1}{n}\sum_{i=1}^n  f({\mathbf X}_1,\, Z_i)-\frac{1}{n}\sum_{i=1}^n  f({\mathbf X}_2,\, Z_i)\right\vert\nonumber
\\\leq&\,\max_{\substack{{\mathbf X}_1,\,{\mathbf X}_2\in\mathcal S_p
\\\cap \{{\mathbf X}:\,\sigma_{\max}({\mathbf X})\leq R,\\
\,\Vert\mathbf X_1-\mathbf X_2\Vert\leq \tau\}
}} \frac{1}{n}\sum_{i=1}^n\Vert  f({\mathbf X}_1,\, Z_i)- f({\mathbf X}_2,\, Z_i)\Vert\nonumber
\\\leq&\,\max_{\substack{{\mathbf X}_1,\,{\mathbf X}_2\in\mathcal S_p
\\\cap \{{\mathbf X}:\,\sigma_{\max}({\mathbf X})\leq R,\\
\,\Vert\mathbf X_1-\mathbf X_2\Vert\leq \tau\}
}}\frac{1}{n}\sum_{i=1}^n \mathcal C(Z_i)\Vert {\mathbf X}_1-{\mathbf X}_2\Vert\leq (2K_C+ \mathcal  C_\mu)\cdot \tau\nonumber
\end{align} 
We have the desired result by combining the above with \eqref{prob lipschitz bound}.
  \end{proof}
  
   \begin{lemma}\label{expected cost result Matrix}
Under Assumption \ref{Lipschitz condition},  for all $${\mathbf X}_1,\,{\mathbf X}_2\in \mathcal S_p:\,\,\max\{\sigma_{\max}({\mathbf X}_1),\,\sigma_{\max}({\mathbf X}_2)\}\leq R,$$
it holds that
\begin{align}
\left\vert \mathbb E[ \mathcal F_n({\mathbf X}_1, \mathbf Z_1^n)]-\mathbb E[\mathcal F_n({\mathbf X}_2, \mathbf Z_1^n)]\right\vert\leq \mathcal  C_\mu\cdot \Vert {\mathbf X}_1-{\mathbf X}_2\Vert.
\end{align}
\end{lemma}
 \begin{proof}  This proof follows a closely similar lemma by \cite{Shapiro:book}. Again, a similar proof has also been provided by \cite{Liu2019}, but some subtle differences make it necessary to conduct the repetition herein.
As per Assumption \ref{Lipschitz condition}, it holds that
\begin{align}
\mathbb E\left[\vert \mathcal  F_n({\mathbf X}_1, \mathbf Z_1^n)-\mathcal  F_n({\mathbf X}_2, \mathbf Z_1^n)\vert\right]\leq \mathbb E\left[ \sum_{i=1}^n\frac{\mathcal C(Z_i)}{n}\Vert{\mathbf X}_1-{\mathbf X}_2\Vert\right]. \nonumber
\end{align}
Due to the convexity of the function $\vert\cdot\vert$, it therefore holds that
\begin{align}
\left\vert \mathbb E\left[\mathcal F_n({\mathbf X}_1, \mathbf Z_1^n)\right]-\mathbb E\left[\mathcal F_n({\mathbf X}_2, \mathbf Z_1^n)\right]\right\vert\leq&\,\mathbb E\left[ \sum_{i=1}^n\frac{\mathcal C(Z_i)}{n}\Vert{\mathbf X}_1-{\mathbf X}_2\Vert\right] \nonumber
\\=& \, \mathbb E\left[ \sum_{i=1}^n\frac{\mathcal C(Z_i)}{n}\right]\cdot\Vert{\mathbf X}_1-{\mathbf X}_2\Vert.\nonumber
\end{align}
Invoking Assumption \ref{Lipschitz condition} again, it holds that $\mathbb E\left[\sum_{i=1}^n\frac{\mathcal C(Z_i)}{n}\right]= \frac{\sum_{i=1}^n\mathbb E[\mathcal C(Z_i)]}{n}\leq \mathcal  C_\mu$ for all $i=1,...,n$, which immediately leads to the desired result.
  \end{proof}
  
  \begin{lemma}\label{initial gap}
Denote
that $ {\mathbf X}^{\ell_1}_{\lambda}\in\underset{{\mathbf X}\in\mathcal S^p}{\arg\,\min}\, \mathcal F_{n}(\mathbf X,\, \mathbf Z_1^n)+\lambda\left\Vert \mathbf X \right\Vert_{*},$
it holds that  $\mathcal F_{n,\lambda}( {\mathbf X}^{\ell_1}_{\lambda},\,\mathbf Z_1^n)\leq  \mathcal F_{n,\lambda}({\mathbf X}^*,\, \mathbf Z_1^n)+\lambda \Vert {\mathbf X}^* \Vert_*$.
\end{lemma}
 \begin{proof}  
 This proof generalizes a similar one in \cite{Liu2019} from sparsity-inducing penalty to low-rankness-inducing penalty; that is, from $\ell_1$ regularization to nuclear norm-based regularization.
As per Assumption \ref{Lipschitz condition}, it holds that
We first invoke the definition of $P_\lambda$ to obtain
  \begin{align}
0\leq P_\lambda(t)=\int_0^{t}\frac{[a\lambda-\theta]_+}{a}d\theta\leq\int_0^{t}\frac{a\lambda}{a}d\theta=\lambda\cdot t.\label{By MCP definition here}
\end{align}
for all $t\geq 0$.
Secondly, by the definition of $ {\mathbf X}^{\ell_1}_{\lambda}$,  
\begin{align}
 \mathcal F_{n}( {\mathbf X}^{\ell_1}_{\lambda},\, \mathbf Z_1^n)+\lambda\Vert    {\mathbf X}^{\ell_1}_{\lambda} \Vert_*\leq  \mathcal F_{n}({\mathbf X}^{*},\, \mathbf Z_1^n)+\lambda\Vert {\mathbf X}^{*} \Vert_*.\label{small lemma first inequality}
\end{align}
Combining \eqref{By MCP definition here} and \eqref{small lemma first inequality}, it holds that
\begin{align}
&\, \mathcal F_{n}( {\mathbf X}^{\ell_1}_{\lambda},\, \mathbf Z_1^n)+\sum_{j=1}^pP_\lambda\left(\vert \sigma_j( {\mathbf X}^{\ell_1}_{\lambda})\vert\right) \nonumber
 \leq \, \mathcal F_{n}( {\mathbf X}^{\ell_1}_{\lambda},\, \mathbf Z_1^n)+\sum_{j=1}^p\lambda\cdot \vert \sigma_j( {\mathbf X}^{\ell_1}_{\lambda})\vert \nonumber
 \\\leq &\, \mathcal F_{n}({\mathbf X}^{*},\, \mathbf Z_1^n)+ \sum_{j=1}^pP_\lambda\left(\vert \sigma_j({\mathbf X}^*)\vert\right)+\lambda\Vert {\mathbf X}^{*} \Vert_*, \nonumber
\end{align}
as desired.
  \end{proof}
  
  \begin{lemma}\label{low rank epsilon net}
  Let $S_{r,R}:=\{X\in\Re^{p\times p}:\,{\bf rk}(X)\leq r,\,\sigma_{\max}(X)\leq R \}$. Then, in terms of the Frobenius norm, there exists an $\epsilon$-net $\bar S_r$ obeying $\vert \bar S_r\vert\leq  \left(\frac{9\sqrt{r}R}{\epsilon}\right)^{(2p+1)r}$.
  \end{lemma}
  \begin{proof}The proof follows a closely similar result by \cite[][Lemma 3.1]{Candes2011}. Denote by $X:=U\Sigma V^\top$ the singular value decomposition (SVD)  of a matrix in $S_{r,R}$. Let $D$ be the set of rank-$r$  diagonal matrices with nonnegative diagonal entries and nuclear norm smaller than $R$, and thus any matrix within set $D$ has the Frobenius norm smaller than $\sqrt{r}\cdot R$. We take $\bar D$ be an $\frac{\epsilon}{3}$-net (in terns of Frobenius norm) for $D$ with $\vert \bar D\vert\leq \left(\frac{9\sqrt{r}R}{\epsilon}\right)^r$. 
  
  Let $O_{p,r}:=\{U\in\Re^{p\times r}:\,U^\top U=I\}$. For the convenience of analysis on $O_{p,r}$, we may as well consider $\hat Q_{p,r}:=\{ X\in\Re^{p\times r}:\,\Vert X\Vert_{1,2}\leq 1\}$ and $\Vert X\Vert_{1,2}=\max_j\Vert X_j\Vert$, where $X_j$ denotes the $j$th column of $X$. Verifiably, $ O_{p,r}\subset \hat Q_{p,r}$. We may create an $\frac{\epsilon}{3\sqrt{r}R}$-net for $\hat Q_{p,r}$, denoted by  $\bar O_{p,r}$, which satisfies that $\vert\bar{O}_{p,r}\vert\leq (9\sqrt{r}R/\epsilon)^{pr}$.
  
  For any $X\in S_{r,R}$, one may decompose $X$ and obtain $X=U\Sigma V^\top$. There exists $\bar X=\bar U\bar\Sigma\bar V^\top\in\bar S_{r,R}$ with $\bar U,\,\bar V\in \bar O_{p,r}$, and $\bar \Sigma\in\bar D$ such that $\Vert U-\bar U\Vert_{1,2}\leq \epsilon/(3\sqrt{r}R)$, $\Vert V-\bar V\Vert_{1,2}\leq \epsilon/(3\sqrt{r}R)$, and $\Vert \Sigma-\bar\Sigma\Vert_F\leq \epsilon/3$. This gives $\Vert X-\bar X\Vert_F=\Vert U\Sigma V^\top-\bar U\bar \Sigma\bar V^\top\Vert_F=\Vert U\Sigma V^\top-\bar U\Sigma V^\top+\bar U\Sigma V^\top-\bar U\bar\Sigma V^\top+\bar U\bar \Sigma V^\top-\bar U\bar\Sigma\bar V^\top \Vert_F\leq \Vert (U-\bar U)\Sigma V^\top\Vert_F+\Vert\bar U(\Sigma-\bar\Sigma) V^\top\Vert_F+\Vert\bar U\bar \Sigma(V-\bar V)\Vert_F$. Since $V$ is orthonormal matrix, $\Vert (U-\bar U)\Sigma V^\top\Vert_F=\Vert (U-\bar U)\Sigma \Vert_F=\sqrt{\sum_{1\leq j\leq r}[\sigma_j(X)]^2\cdot\Vert \bar U_j- U_j\Vert_2^2}\leq   \sqrt{ \Vert \Sigma\Vert_F^2\cdot \Vert U-\bar U\Vert_{1,2}^2}\leq \epsilon/3$, where $U_j$ is the $j$th column of $U$.  By a symmetric argument, we may also obtain that $\Vert\bar U\bar \Sigma(V-\bar V)^\top\Vert_F\leq   \epsilon/3$. To bound the second term, we also notice that $\Vert\bar U(\Sigma-\bar\Sigma) V^\top\Vert_F=\Vert\Sigma-\bar\Sigma\Vert_F\leq \epsilon/3$. Combining the above provides the desired result.
  \end{proof}
\end{appendix}



\end{document}